\documentclass[12pt,leqno]{article}
\usepackage{amsmath, amsthm, amsfonts, amssymb, color}
\usepackage{mathrsfs}
\usepackage{color}

\newtheorem{thm}{Theorem}[section]

\theoremstyle{definition}

\newcommand{\scr}[1]{\mathscr #1}
\definecolor{wco}{rgb}{0.5,0.2,0.3}

\numberwithin{equation}{section} \theoremstyle{remark}

\newcommand{\ua}{\uparrow}

\newcommand{\Pp}{\mathbb P}

\newcommand{\nnu}{\mathsf{n}}
\title{{\bf Coupling and Strong Feller  for Jump Processes on Banach Spaces}\footnote{Supported in
 part by NNSFC(11131003, 11126350 and 11201073), SRFDP, the Fundamental Research Funds for the Central Universities and the Programme of Excellent Young Talents in Universities
of Fujian (JA10058 and JA11051).}
}
\author{
{\bf Feng-Yu Wang$^{a),c)}$ and Jian Wang$^{b)}$}\\
\footnotesize{$^{a)}$School of Mathematical Sciences,
Beijing Normal University, Beijing 100875, China}\\
\footnotesize{$^{b)}$School of Mathematics and Computer Science, Fujian Normal
University, Fuzhou 350007, China.}\\ \footnotesize{$^{c)}$Department of Mathematics,
Swansea University, Singleton Park, SA2 8PP, UK}\\
\footnotesize{Email: wangfy@bnu.edu.cn; F.Y.Wang@swansea.ac.uk; jianwang@fjnu.edu.cn}}
\begin{document}
\def\R{\mathbb R}  \def\ff{\frac} \def\ss{\sqrt} \def\B{\mathbf
B}
\def\N{\mathbb N} \def\kk{\kappa} \def\m{{\bf m}}
\def\dd{\delta} \def\DD{\Delta} \def\vv{\varepsilon} \def\rr{\rho}
\def\<{\langle} \def\>{\rangle} \def\GG{\Gamma} \def\gg{\gamma}
  \def\nn{\nabla} \def\pp{\partial} \def\EE{\scr E}
\def\d{\text{\rm{d}}} \def\bb{\beta} \def\aa{\alpha} \def\D{\scr D}
  \def\si{\sigma} \def\ess{\text{\rm{ess}}}
\def\beg{\begin} \def\beq{\begin{equation}}  \def\F{\scr F}
\def\Ric{\text{\rm{Ric}}} \def\Hess{\text{\rm{Hess}}}
\def\e{\text{\rm{e}}} \def\ua{\underline a} \def\OO{\Omega}  \def\oo{\omega}
 \def\tt{\tilde} \def\Ric{\text{\rm{Ric}}}
\def\cut{\text{\rm{cut}}} \def\P{\mathbb P} \def\ifn{I_n(f^{\bigotimes n})}
\def\C{\scr C}      \def\aaa{\mathbf{r}}     \def\r{r}
\def\gap{\text{\rm{gap}}} \def\prr{\pi_{{\bf m},\varrho}}  \def\r{\mathbf r}
\def\Z{\mathbb Z} \def\vrr{\varrho} \def\ll{\lambda}
\def\L{\scr L}\def\Tt{\tt} \def\TT{\tt}\def\II{\mathbb I}
\def\i{{\rm in}}\def\Sect{{\rm Sect}}\def\E{\mathbb E} \def\H{\mathbb H}
\def\M{\scr M}\def\Q{\mathbb Q} \def\texto{\text{o}} \def\LL{\Lambda}
\def\Rank{{\rm Rank}} \def\B{\scr B} \def\i{{\rm i}} \def\HR{\hat{\R}^d}
\def\BB{\mathbb B}\def\vp{\varphi}

\maketitle
\begin{abstract} By using lower bound conditions of the L\'evy measure w.r.t. a nice reference measure,
the coupling and strong Feller properties   are investigated for the Markov semigroup associated with  a class of linear SDEs driven by (non-cylindrical) L\'evy
processes on a Banach space. Unlike in the finite-dimensional case where   these properties have also been confirmed for L\'evy processes  without drift,
in the infinite-dimensional setting the appearance of a drift term is essential   to ensure the quasi-invariance of the process
 by   shifting the initial data. Gradient estimates and exponential convergence are also investigated. The main results are illustrated by specific models on
the Wiener space and  separable Hilbert spaces.
\end{abstract} \noindent

 \vskip 2cm

\section{Introduction}

In recent years,   the coupling property, the strong Feller property, and gradient estimates
have been intensively investigated for linear stochastic differential equations driven by L\'evy processes on $\R^d$, see e.g. \cite{PZ, W10a, W10b, SW, BSW, SSW, SW2, KS2, KS} and references within. In these references the shift-invariance of the Lebesgue measure plays an essential role.  When the state space is  infinite-dimensional so that the Lebesgue measure is no longer available,   we need a reference measure which is quasi-invariant under a reasonable class of shift transforms. Typical
examples of the reference measure include    the Wiener measure on the continuous path space and the Gaussian measure on a  Hilbert space, see Section 5 for details. The purpose of this paper is to investigate regularity properties of linear SDEs
driven by L\'evy processes on a Banach space equipped with such a nice reference measure. To ensure the quasi-invariance of the solution,
a  strong enough linear drift term will be needed.

On the other hand, concerning (semi-)linear SDEs on Hilbert spaces, when the noise is a cylindrical $\aa$-stable process, many regularity results derived in finite dimensions can be extended to the infinite-dimensional setting (see \cite{X1,X2,WJ}); and  when the noise has  a non-trivial Gaussian part, the regularity properties can be  derived by using the drift part and the Gaussian part (see e.g. \cite{Z, DZ, BN, RW03}).   But there seems to be no    results concerning the strong Feller    and   coupling properties  for   SDEs   driven by purely jump non-cylindrical L\'evy processes. In this paper we intend to investigate these properties for linear SDEs driven by non-cylindrical L\'evy noise on Banach spaces.

Let $(\BB, \|\cdot\|_\BB)$ be a Banach space and let $\mu$ be a probability measure on $\BB$ having full support. Let $\BB'$ be the dual space of $\BB$ with $\langle \cdot, \cdot\rangle$  the duality between $\BB$ and $\BB'$. Let $(\H,\|\cdot\|_\H)$ be another Banach space which is densely and continuously embedded into  $ \BB$ such that for any $h\in\H$, $\mu$ is quasi-invariant under the shift $x\mapsto x+h$; that is, there exists a non-negative measurable function $\vp_h$ on $\BB$ such that
\beq\label{1.1} \mu(\d z-h)= \vp_h(z)\,\mu(\d z).\end{equation} Let $L_t$ be a L\'evy process on $\BB$ with L\'evy measure $\nu$.   Recall that a $\si$-finite measure $\nu$ on $\BB$ is called a L\'evy measure if $ \nu (\{0\})=0$ and the mapping from $\BB'$ to $\R$ given by
$$\BB'\ni a\mapsto \exp\bigg[\int_\BB\big(\cos\,\<x,a\>-1\big)\, \nu  (\d x)\bigg]$$ is the characteristic function of a random variable on $\BB$. Note that
since $\cos$ is an even function, one may replace $\nu$ by the symmetric measure $\nu+\nu^*$ as  in \cite{App1}, where $\nu^{*}(A)=\nu(-A)$ for any $A\in\B$.  When $\BB$ is a Hilbert space,  $\nu$ is a L\'evy measure if and only if  $\nu(\{0\})=0$ and  $\int_\BB (1\land\|x\|_\BB^2)\,\nu(\d x)<\infty$; while in general, $\nu$ is a L\'evy measure provided  $\nu(\{0\})=0$ and  $\int_\BB (1\land\|x\|_\BB )\,\nu(\d x)<\infty$ (see \cite{App0,App1}).

 Let $\si: \BB\to\BB$ be a bounded linear operator  and let $(A,\D(A))$ be a  linear  operator on $\BB$ generating a $C_0$ semigroup $(T_s)_{s\ge 0}$.
 Consider the following linear SDE on $\BB$:
 \beq\label{E1} \d X_t= AX_t\,\d t + \si \,\d L_t.\end{equation} For any $x\in \BB$, the solution with initial data $x$ is
\beq\label{S} X_t^x= T_t x +\int_0^t T_{t-s} \si\,\d L_s,\ \ \ t\ge 0.\end{equation} See \cite{CM, PZ1, App0, App1} for the detailed construction of this solution.
 Let $\B_b(\BB)$ be the class of all bounded measurable functions on $\BB$. We aim to investigate the coupling property and the strong Feller property for the associated Markov semigroup
 $$P_t f(x):= \E f(X_t^x),\ \ \ t\ge 0, x\in \BB, f\in \B_b(\BB).$$
 Recall that the solution has successful coupling if and only if  (cf. \cite{Lin,CG})
 $$\lim_{t\to \infty} \|P_t(x,\cdot)- P_t(y,\cdot)\|_{var}=0,\ \ \ x,y\in\BB,$$ where $P_t(x,\d y)$ is the transition kernel of $P_t$ and $\|\cdot\|_{var}$ is the total variation norm. Let $\rr_0$ be a non-trivial non-negative measurable function   on $\BB$ such that
\beq\label{C} \nu(\d z)\ge \rr_0(z)\,\mu(\d z)=:\nu_0(\d z)\end{equation} holds. Thus, the L\'evy process considered here is essentially different from the cylindrical $\aa$-stable process used in \cite{X1,X2}. Indeed, for $\BB$ being a  Hilbert space with ONB $\{e_i\}_{i\ge 1}$, the L\'evy measure (if exists) for a cylindrical L\'evy process is supported on $\cup_{i\ge 1} \R e_i$ and hence, is singular w.r.t. e.g. a non-trivial Gaussian probability measure $\mu$.    Assume
 \paragraph{(A)}  {\rm Ker}$(\si)=\{0\}$ and $T_s \BB\subset \si \H$ holds for any $s>0.$

 \

 Obviously, {\bf(A)} implies that for any $s>0$,  the operator  $\si^{-1}T_s: \BB\to\H$ is well defined.

 \beg{thm}\label{T1.1} Assume {\bf (A)}. Suppose that $\nu_0$ in $(\ref{C})$ is infinite; i.e.\ $\nu_0(\BB)=\infty.$  \beg{enumerate} \item[$(1)$]  If for any $h\in\H$
\beq\label{N}\sup_{\vv\in (0,1)} \vp_{\vv h}(\cdot+\vv h)<\infty,\ \  \mu\text{-a.e.},\end{equation} then for any $f\in \B_b(\BB)$ and $t>0$, $P_tf$ is directionally continuous; i.e.
$\lim_{\vv\to 0} P_t f(x+\vv y)= P_t f(x)$ holds for any $x,y\in\BB.$
\item[$(2)$] If for any $s>0$
\beq\label{1.5'}\sup_{\|y\|_\BB\le 1} \vp_{\si^{-1}T_s y}(\cdot+\si^{-1}T_s y)<\infty,\ \  \mu\text{-a.e.},\end{equation}
then $P_t$ is strong Feller for $t>0$; i.e. $P_t \B_b(\BB)\subset C_b(\BB).$ \end{enumerate} \end{thm}

A simple example for $\nu_0(\BB)=\infty$ to hold is as follows. Let $z\to \|z\|_\BB$ have a strictly positive distribution density function $\rr$ under the probability measure $\mu$, for instance it is the case when $\mu$  is the Wiener measure (see Subsection 5.1 below).   Let $r_0\in (0,\infty]$, and let $\aa\in (0,2)$ when $\BB$ is a Hilbert space and $\aa\in (0,1)$ otherwise. Then
$$\nu_0(\d z):= \ff{1_{(0,r_0)}(\|z\|_\BB)}{\rr(\|z\|_\BB)\|z\|^{1+\aa}_\BB}\,\mu(\d z)$$ is a L\'evy measure on $\BB$ with $\nu_0(\BB)=\infty.$
This measure is an infinite-dimensional version of the $\aa$-stable jump measure. Modifying arguments from   \cite[Theorem 3.1]{W10a} and \cite[Theorem 1.1]{SW2} where the coupling property has been investigated in the finite-dimension setting, we have the following two assertions on the coupling property with estimates on the convergence rate.
For $r>0$ and $z\in\BB$, let  $B(z,r)=\{y\in \BB: \|z-y\|_\BB<r\}$ be the open ball at $z$ with radius $r$.

 \beg{thm}\label{T1.2}  Assume {\bf (A)}. Suppose that    $\nu_0$   in $(\ref{C})$ is finite; i.e.\ $\nu_0(\BB)<\infty$, $\si$ is invertible with $\|\si^{-1}\|_\BB<\infty$, and   $\|T_s\|_{\BB}\le c$  holds for some constant $c>0$ and all $s>0.$
\begin{itemize}
\item[{\rm (i)}]  If there exist $z_0\in\BB$ and $r_0>0$ such that
 \beq\label{C2} \dd_1(\vv):=\sup_{s\ge \vv, \|x\|_\BB\le 1} \int_{B(z_0, r_0)} \ff{\vp_{\si^{-1}T_s x}(z)^2\rr_0(z-\si^{-1}T_sx)^2}{ \rr_0(z) }\,\mu(\d z)<\infty,\ \ \vv>0,\end{equation} then there exists a constant $C>0$   such that
 \beq\label{Coupling} \|P_t(x,\cdot)-P_t(x+y,\cdot)\|_{var} \le C (1+\|y\|_\BB) \inf_{\vv\in (0,1)} \bigg(\vv +\sqrt{\ff{ {\dd_1(\vv)}}{ t}}\bigg),\ \ t>0,\ x,y\in\BB \end{equation}  holds.

 \item[{\rm (ii)}]If there exist $z_0\in \BB$ and $r_0>0$ such that  \beq\label{C22}\dd_2(\vv):=\sup_{s\ge \vv, \|x\|_{\BB}\le 1} \int_{B(z_0, r_0)} \frac{\varphi_{\sigma^{-1}T_s x}(z)^2\vee 1}{\rho_0(z)}\,\mu(dz)<\infty,\ \ \vv>0,\end{equation} then there exist two constants $C>0 $   such that for all $x,y\in\BB$ and $t>0$,
\begin{equation}\label{Coupling1}\begin{aligned}
   \|P_t(x,\cdot)-P_t(y,\cdot)\|_{var}\le
  C (1+\|x-y\|_{\BB})\inf_{\varepsilon\in (0,1)}\!\!\bigg(\varepsilon+  \sqrt{\frac{{\dd_2(\vv)}}{{t}}} \bigg).\end{aligned}
\end{equation}

 \end{itemize}
 \end{thm}
Using $\rr_0\land 1$ in place of $\rr_0$, one may   replace \eqref{C2} by
$$\tt\dd_1(\vv):=\sup_{s\ge \vv, \|x\|_\BB\le 1} \int_{B(z_0, r_0)} \ff{\vp_{\si^{-1}T_s x}(z)^2}{1\land \rr_0(z)}\,\mu(\d z)<\infty,\ \ \vv>0.$$
If $\inf_{z\in B(x_0,r_0)}\rho_0(z)>0$, then this condition and \eqref{C22} are equivalent. But in general (\ref{C2}) and (\ref{C22}) are incomparable. Next, it is easy to see  that the convergence rate implied by (\ref{Coupling}) or \eqref{Coupling1} is in general slower than $\ff 1 {\ss t}$. Our next result shows that if $\vp$ and $\rr_0$ are regular enough,   the convergence could be exponentially fast.

\beg{thm}\label{T1.3}  Assume {\bf (A)}. Suppose that $\nu_0$ in $(\ref{C})$ is finite with $\ll_0:=\nu_0(\BB)\in (0,\infty)$, $\|T_s\|_\BB\le c\e^{-\ll s}$   and
\beq\label{Z1} \int_{\BB} \Big(|\rr_0(z)-\rr_0(z+h)|+\rr_0(z)|\vp_h(z)-1|\Big)\,\mu(\d z)\le c\|h\|_\H,\ \ \|h\|_\H\le 1\end{equation} holds  for some constants $c,\ll>0$ and all $s\ge 0$. If
\beq\label{Z2} \sup_{t\ge 1}\ff 1 {1-\e^{-\ll_0 t}} \int_0^t\e^{-\ll_0 r} \Big(\sup_{\|z\|_\BB\le 1}\sup_{s\ge r}\|\si^{-1}T_sz\|_\H\Big)\,\d r <\infty,\end{equation} then there exists a constant $C>0$ such that
\beq\label{Z3} \|P_t(x,\cdot)-P_t(y,\cdot)\|_{var} \le C (1+\|x-y\|_\BB)\e^{-\ff{\ll_0\ll t}{\ll_0+\ll}},\ \ x,y\in \BB, t\ge 0.\end{equation} \end{thm}

Following the line of \cite[Section 3]{W11}, one may  also naturally investigate gradient estimates and derivative formula for $P_t$. It is not difficult to present a formal result under a condition similar to \cite[(3.1)]{W11}, for instance:
\beg{prp}\label{PP0} Assume that $\{h\in\H: \sup_{s\in [0,1]}\|\si^{-1}T_sh\|_\H<\infty\}$ is dense in $\BB$. If there exists a non-negative function $g$ on $\BB$ such that $ \nu_0(\{g>0\})=\infty$, $\rr_0g$ is bounded and Lipschitz continuous in $\|\cdot\|_\H$, and
  \beq\label{T1-2} \beg{split} q(t)&:=\sup_{\|h\|_\H\in (0,1]}\bigg\{\Big(1+ \ff{\mu(|\vp_h-1|)}  {\|h\|_\H}\Big)\int_0^\infty \e^{-t\nu_0(1-\exp[-rg])}\,\d r \\
  &\qquad\qquad\qquad \quad+ \ff{\mu\big(|g-g(\cdot-h)|\big)}  {\|h\|_\H}\int_0^\infty r\e^{-t\nu_0(1-\exp[-rg])}\,\d r\bigg\} <\infty,\ \ t>0,\end{split}\end{equation}
 then there exists a constant $C_1>0$ such that \beg{equation*}\beg{split} |\nn_y P_t f(x)|:&=\,\limsup_{\vv\downarrow 0} \ff 1\vv |P_tf(x+\vv y)-P_tf(x)|\\
 &\le \,C_1\|f\|_\infty q(t) \int_0^t \|\si^{-1}T_sy\|_\H\,\d s,\ \ f\in\B_b(\BB), t>0, x,y\in \BB.\end{split}\end{equation*}
  Suppose moreover that $\|T_s\|_\BB\le c\e^{-\ll s}$ for some constants $c,\ll>0$ and all $s\ge 0$. Then
 $$\|P_t(x,\cdot)-P_t(y,\cdot)\|_{var} \le C_2 (1+\|x-y\|_\BB)\e^{- \ll t },\ \ x,y\in \BB, t\ge 0$$ holds for some constant $C_2>0.$
 \end{prp}

Unfortunately, in the moment  we do not have any non-trivial example in infinite dimensions to illustrate condition (\ref{T1-2}). Indeed, it seems that in infinite dimensions  the uniform norm of the gradient of $P_t$
$$\|\nn P_t\|_\infty:= \sup\{|\nn_y P_t f(x)|:\ \|y\|_\BB\le 1, x\in \BB, \|f\|_\infty\le 1\}$$ is  most likely infinite for any $t>0.$ The intuition is that comparing with   a cylindrical noise given in \cite[Assumption 2.2]{X2}, which is strong enough along   single directions so that the noise might not take values in $\BB$,  our non-cylindrical L\'evy process   seems too weak to imply a bounded gradient estimate of $P_t$. Nevertheless, we are able to estimate the uniform gradient of a modified version of $P_t$ (cf. Proposition \ref{PP2} below), which implies the desired exponential convergence in (\ref{Z3}).

We remark that the derivative formula and gradient estimate are investigated in \cite{T,Zhang} for SDEs on $\R^d$ driven by L\'evy noises, where in \cite{T} the process may contain   a diffusion part but extensions of  the main results   to  infinite dimensions are not yet available, while  in \cite{Zhang} the main result was also extended to a class of semi-linear SPDEs driven by cylindrical $\aa$-stable processes. Both papers are quite different from the present one, where we aim to describe regularity properties of the semigroup merely using the L\'evy measure of the noise.

\

We will prove Theorems \ref{T1.1} (also Proposition \ref{PP0}), \ref{T1.2} and \ref{T1.3}  in the following three sections respectively. In Section 5 we present two specific examples, with $\mu$ the Wiener measure on a Brownian path space and  the Gaussian measure on an Hilbert space respectively, to illustrated these results.

\section{Proofs of Theorem \ref{T1.1} and Proposition \ref{PP0}}

The key technique of the study is the coupling by change of measure. For readers' convenience, let us briefly recall the main idea of the argument.  To  investigate e.g. the continuity of $P_tf$ along $y\in \BB$, for any $x\in\BB$ we construct a family of processes $\{X_\cdot^\vv\}_{\vv\in [0,1)}$ and the associated probability densities $\{R_\vv\}_{\vv\in {[0,1)}}$    such that
\beg{enumerate} \item[(1)] $X_0^\vv = x+\vv y,\ X_t^\vv= X_t^0,\ \vv\in [0,1),\ t>0;$
\item[$(2)$] Under the probability $R_\vv \P$, the process $X_\cdot^\vv$ is associated to the transition semigroup $(P_s)_{s\ge 0};$
\item[$(3)$] $\lim_{\vv\to 0} R_\vv =R_0=1$ holds in $L^1(\P).$ \end{enumerate} Then, for any bounded measurable function $f$ and $t>0$,
$$\lim_{\vv\to 0} P_t f(x+\vv y)= \lim_{\vv\to 0} \E\big[R_\vv f(X_t^\vv)\big] =\lim_{\vv\to 0} \E\big[R_\vv f(X_t^0)\big] = \E \big[R_0 f(X_t^0)\big] = P_t f(x).$$
To realize this idea in the present setting, the following Lemma \ref{L2.1} will play a crucial role.

For fixed  $t>0$, let $\LL$ be the distribution of $L:=(L_s)_{s\in [0,t]}$ which is a probability measure on the paths pace
\beg{equation*}\beg{split} W_t=\big\{w:  [0,t]\to \BB\ &\text{is\  right-continuous\ having left limits}\big\} \end{split}\end{equation*}equipped with the Skorokhod metric.
  For any $w\in W_t$,  let
$$w(\d z,\d s):= \sum_{s\in [0,t], \DD w_s \ne 0}  \dd_{(\DD w_s, s)},$$ which  records jumps of the path $w$, where $\DD w_s= w_s-w_{s-}$. Let
$$w(g)=\int_{\BB\times [0,t]}g(z,s)\,w(\d z,\d s)=\sum_{s\in [0,t], \DD w_s\ne 0} g(\DD w_s, s),\ \ \ g\in L^1(w).$$
A function $g$ on $\BB$ will be also regarded as a function on $\BB\times [0,t]$ by letting $g(z,s)=g(z)$ for $(z,s)\in\BB\times [0,t].$

Moreover,  write $L=L^1+L^0$, where $L^1$ and $L^0$ are two independent L\'evy processes
with L\'evy measure $\nu-\nu_0$ and $\nu_0$ respectively, and $L^0$ does not have a Gaussian term.
Let $\LL^1$ and $\LL^0$ be the distributions of $L^1$ and $L^0$ respectively. We have $\LL=\LL^1*\LL^0.$

Repeating the proof of  \cite[Lemma 2.1]{W11} where $\BB=\R^d$, we have the following result.

\beg{lem}\label{L2.1} For any $h\in L^1(W_t\times\BB\times [0,t]; \LL^0\times \nu_0\times \d s)$,
\beq\label{F1}\beg{split} &\int_{W_t\times\BB\times [0,t]}  h(w,z,s)\, \LL^0(\d w)\,\nu_0(\d z)\,\d s\\
&=\int_{W_t}\LL^0(\d w)\int_{\BB\times [0,t]} h(w-z1_{[s,t]},z, s)\,w(\d z,\d s).\end{split}\end{equation}\end{lem}
To prove Theorem \ref{T1.1}, we also need the following two more lemmas.

\beg{lem}\label{L2.2} Let $y\in\BB$ such that $\si^{-1}T_s y\in\H$ for any $s>0$, and let $g$ be a non-negative measurable function on $\BB$ such that
$\nu_0(g):=\int_\BB g\,\d\nu_0 <\infty$ and $w(g)>0$ for $\LL^0$-a.e. $w$. Let
  $$\Phi_\vv(w,z,s)= \ff{ \vp_{\vv \si^{-1}T_s y}(z)(\rr_0 g)(z-\vv \si^{-1}T_s y)}{w(g)
+ g(z-\vv \si^{-1}T_sy)},\ \ \vv\ge 0.$$  If $(\ref{N})$ holds for any $h\in \H$, then $\{\Phi_\vv\}_{\vv\in [0,1)}$ is uniformly integrable w.r.t.
$\LL^0\times\mu\times \d s$  on $W_t\times \BB\times [0,t].$ \end{lem}
\beg{proof} Since $\vp_0\equiv 1$, applying (\ref{F1}) to $h(w,z,s)= \ff{g(z)}{w(g)}$ we obtain
\beg{equation}\label{WW1}\beg{split} &\int_{W_t \times\BB\times [0,t]} \Phi_0(w,z,s)\,\LL^0(\d w)\,\mu(\d z)\,\d s\\
 &= \int_{W_t \times\BB\times [0,t]} \ff{g(z)}{w(g)+g(z)}\, \LL^0(\d w)\,\nu_0(\d z)\,\d s\\
&= \int_{W_t} \LL^0(\d w)\int_{\BB\times [0,t]} \ff{g(z)}{w(g)}\,w(\d z,\d s)\\
&= 1.\end{split}\end{equation}
Next,
  by (\ref{1.1}) and the integral transform
$z\mapsto z-\vv\si^{-1}T_sy$, for any $F\in \B_b(W_t\times \BB\times [0,t])$ we have
\beg{equation}\label{WW}\beg{split} &\int_{W_t\times\BB\times [0,t]} F(w, z+\vv\si^{-1}T_s y, s)\Phi_0(w,z,s)\,\LL^0(\d w)\,\mu(\d z)\,\d s \\
&= \int_{W_t\times\BB\times [0,t]} \ff{F(w, z+\vv\si^{-1}T_s y, s)(\rr_0g)(z)}{w(g)+g(z)}\,\LL^0(\d w)\,\mu(\d z)\,\d s \\
&=\int_{W_t\times\BB\times [0,t]} F(w, z, s)\Phi_\vv(w,z,s)\,\LL^0(\d w)\,\mu(\d z)\,\d s.\end{split}\end{equation} Letting $F=1$ and combining this with (\ref{WW1}), we conclude that $\{\Phi_\vv\}_{\vv\in [0,1)}$ are probability densities w.r.t. $\LL^0\times \mu\times \d s$. Moreover,
applying (\ref{WW}) to $F(w,z,s)= 1_{\{\Phi_\vv>R\}}$ for $R>0$ and letting
$$\eta(w,z,s)= \sup_{\vv\in (0,1)} \ff{(\rr_0g)(z)}{w(g)+g(z)} \vp_{\vv \si^{-1}T_s y}(z+\vv\si^{-1}T_s y)$$ which is finite
$\LL^0\times \mu\times \d s$-a.e., we obtain
\beg{equation*}\beg{split} &\sup_{\vv\in (0,1)}\int_{W_t\times \BB\times [0,t]} (\Phi_\vv 1_{\{\Phi_\vv>R\}})(w,z,s)\,\LL^0(\d w)\,\mu(\d z)\,\d s \\
&\le\int_{W_t\times \BB\times [0,t]} (\Phi_0 1_{\{\eta >R\}})(w,z,s)\,\LL^0(\d w)\,\mu(\d z)\,\d s \end{split}\end{equation*} which goes to zero as $R\to\infty$ by the dominated
 convergence theorem. \end{proof}

 \beg{lem}\label{L2.3} Let $E$ be a topology space and $C_b(E)$ be the class of all bounded continuous functions on $\BB$. Let $\mu_0$ be a finite measure on the Borel $\si$-field $\B$ such that $C_b(E)$ is dense in
 $L^1(\mu_0)$. Let $\{f_n\}_{n\ge 1}$ be a sequence of uniformly integrable functions w.r.t. $\mu_0$ such that
 $$\lim_{n\to\infty} \int_E (Ff_n)\,\d\mu_0= \int_E (Ff_0)\,\d\mu_0 $$ holds for some $f_0\in L^1(\mu_0)$ and all $F\in C_b(E)$. Then it holds also for any
 $F\in \B_b(E).$ \end{lem}
 \beg{proof} Let $\vv(R)= \sup_{n\ge 1} \mu_0 (|f_n-f_0|1_{\{|f_n-f_0|>R\}})$ which goes to zero as $R\to\infty$. For any $F\in \B_b(E)$, let $\{F_m\}_{m\ge 1}\subset C_b(E)$ such that $\|F_m\|_\infty\le \|F\|_\infty$ and $\mu_0(|F_m-F|)\le \ff 1 m.$ Then
 \beg{equation*}\beg{split} \bigg|\int_E F(f_n-f_0)\,\d\mu_0\bigg|& \le \bigg|\int_E F(f_n-f_0)1_{\{|f_n-f_0|\le R\}}\,\d\mu_0\bigg|+\|F\|_\infty\vv(R)\\
 &\le \bigg|\int_EF_m(f_n-f_0)1_{\{|f_n-f_0|\le R\}}\,\d\mu_0\bigg|+\|F\|_\infty\vv(R)+\ff R m\\
 &\le \bigg|\int_EF_m(f_n-f_0)\, \d\mu_0\bigg|+2\|F\|_\infty\vv(R)+\ff R m.\end{split}\end{equation*} By first letting $n\to\infty$ then $m\to\infty$ and finally $R\to\infty$, we complete the proof.\end{proof}

\beg{proof}[Proof of Theorem \ref{T1.1}]   (1) Let $f\in \B_b(\BB)$ and $x,y\in\BB$ be fixed. For any $\vv>0$, let
$$F_\vv(w)= f\bigg(T_t (x+\vv y) +\int_0^tT_{t-s}\si\,\d w_s\bigg),$$ where $\int_0^t T_{t-s}\si\,\d w_s$ is the It\^o stochastic integral which is $\LL$-a.e. well-defined. Let e.g.   $g = \ff 1 {\rr_0\lor 1}.$ We have $\nu_0(g)<\infty$ and,
since $\nu_0(\BB)=\infty$ and
$g>0$,   $w(g)>0$ for $\LL^0$-a.e. $w$. Then, by (\ref{S}) and Lemma \ref{L2.1} for
$$h(w^0,z,s)= \ff{F_0(w^1+w^0+(z+\vv\si^{-1}T_sy)1_{[s,t]})g(z)}{w^0(g)+g(z)},$$ we obtain
\beg{equation*}\beg{split} &P_t f(x+\vv y) \\
&= \E F_\vv(L^1+L^0)\\
&= \int_{W_t^2}\LL^1(\d w^1)\,\LL^0(\d w^0)\int_{\BB\times [0,t]} \ff{F_\vv(w^1+w^0) g(z)}{w^0(g)}\,w^0(\d z,\d s)\\
&=\int_{W_t^2}\LL^1(\d w^1)\,\LL^0(\d w^0)\int_{\BB\times [0,t]} \ff{F_0(w^1+w^0+\vv\si^{-1}T_s y 1_{[s,t]})g(z)}{w^0(g)}\,w^0(\d z,\d s)\\
&= \int_{W_t^2}\LL^1(\d w^1)\,\LL^0(\d w^0)\int_{\BB\times [0,t]} \ff{F_0(w^1+w^0+(z+\vv\si^{-1}T_s y) 1_{[s,t]})g(z)}{w^0(g)+ g(z) }\,\nu_0(\d z)\,\d s \\
&= \int_{W_t^2}\LL^1(\d w^1)\,\LL^0(\d w^0)\int_{\BB\times [0,t]} \ff{F_0(w^1+w^0+(z+\vv\si^{-1}T_s y) 1_{[s,t]})(\rr_0g)(z)}{w^0(g)+ g(z) }\,\mu(\d z)\,\d s.\end{split}\end{equation*} Since $\vv \si^{-1}T_s y\in \H$ so that (\ref{1.1}) implies
$$\mu(\d z- \vv \si^{-1} T_s y)= \vp_{\vv\si^{-1}T_s y}(z)\,\mu(\d z),$$  by using the integral transform $z\mapsto z -\vv\si^{-1}T_sy$ and noting that $\LL=\LL^1 *\LL^0$, we obtain
\beg{equation}\label{F4'}\beg{split} & P_t f(x+\vv y)\\
 &= \int_{W_t} \LL(\d w)\int_{\BB\times [0,t]} \ff{F_0(w +z 1_{[s,t]})(\rr_0 g)(z-\vv \si^{-1}T_s y)}{w^0(g)
+ g(z-\vv \si^{-1}T_sy)}\vp_{\vv \si^{-1}T_s y}(z)\,\mu(\d z)\,\d s\\
&=\int_{W_t} \LL^1(\d w^1) \int_{W_t\times\BB\times [0,t]} F_0(w^1+w^0+z1_{[s,t]}) \Phi_\vv(w^0,z,s) \,\LL^0(\d w^0)\,\mu(d z)\,\d s.\end{split}\end{equation}
 Therefore,
 it suffices
to show that
$$   \lim_{\vv\to 0} \int_{W_t\times \BB\times [0,t]}(F  \Phi_\vv)(w,z,s) \,\LL^0(\d w) \, \mu(\d z) \,\d s
  =\int_{W_t\times \BB\times [0,t]}(F  \Phi_0)(w,z,s) \,\LL^0(\d w) \, \mu(\d z)\, \d s $$ holds for any $F\in \B_b(W_t\times \BB\times [0,t]).$  According to (\ref{WW}), this holds provided $F\in C_b(W_t\times\BB\times [0,t]).$ Since the Borel $\si$-field on the Polish space $W_t\times \BB\times [0,t]$ is induced by bounded continuous functions, $C_b(W_t\times \BB\times [0,t])$ is dense in $L^1(\LL^0\times \mu\times \d s)$. Thus, the desired assertion follows from Lemmas \ref{L2.2} and \ref{L2.3}.

  (2) For any sequence $\{y_n\}\subset \BB$ converging to $0$ as $n\to \infty$, define
  $$\Psi_n(w,z,s)= \ff{\varphi_{\si^{-1}T_s y_n}(z) (\rr_0g)(z-\si^{-1}T_sy_n)}{w(g)+g(z-\si^{-1}T_sy_n)},\ \ n\ge 1.$$ Using $\si^{-1} T_s y_n$ to replace $\vv h$ in   the proof of Lemma \ref{L2.2}, we see that (\ref{1.5'}) implies that $\{\Psi_n\}_{n\ge 1}$ is uniformly integrable w.r.t. $\LL^0\times \mu\times \d s$ on $W_t\times\BB\times [0,t]$. Therefore, using $\si^{-1} T_s y_n$ to replace $\vv h$ in the proof of (1), we obtain $\lim_{n\to \infty} P_t f(x+y_n)= P_t f(x)$ for any $f\in \B_b(\BB), t>0$ and $x\in \BB.$
 \end{proof}
 \beg{proof}[Proof of Proposition \ref{PP0}] Since $\{h\in\H: \sup_{s\in [0,1]}\|\si^{-1}T_sh\|_\H<\infty\}$ is dense in $\BB$, it suffices to prove for $y\in\H$ such that $\|\si^{-1}T_sy\|_\H\le 1$ for $s\in [0,1].$
 Since the boundedness of $\rr_0g$ implies $\nu_0(g)<\infty$ and $\nu_0(\{g>0\})=\infty$ implies   $w(g)>0, \LL^0$-a.e., (\ref{F4'}) holds true. By (\ref{F4'}) and (\ref{T1-2}) we have
 \beg{equation}\label{WFY1}\beg{split} &\ff{|P_tf(x+\vv y)-P_tf(x)|}{\vv} \\
 &\le  \ff{\|f\|_\infty}\vv   \int_{W_t\times\BB\times [0,t]} |\Phi_\vv(w,z,s)-\Phi_0(w,z,s)|\,\LL^0(\d w)\,\mu(\d z)\,\d s,\ \ \vv>0.
  \end{split}\end{equation} Since $\rr_0g$ is bounded and Lipschitz continuous in $\|\cdot\|_\H$, there exists a constant $c_1>0$ such that
 \beq\label{WFY2} \beg{split} &|\Phi_\vv(w,z,s)-\Phi_0(w,z,s)|\\
 &\le \ff{|\vp_{\vv \si^{-1}T_s y}(z)-1|}{w(g)} +\bigg|\ff{(\rr_0g)(z-\vv\si^{-1}T_sy)}{w(g)+g  (z-\vv\si^{-1}T_sy)}-\ff{(\rr_0g)(z)}{w(g)+g(z)}\bigg|\\
 &\le  \ff{|\vp_{\vv \si^{-1}T_s y}(z)-1|+c_1\|\vv\si^{-1}T_sy\|_\H}{w(g)} +\ff{c_1|g(z-\vv\si^{-1}T_sy)-g(z)| }{w(g)^2}.\end{split}\end{equation} Moreover,
 according to \cite[Lemma 2.2]{W11} with $\BB$ in place of $\R^d$, for any $\theta>0$, we have
 $$\int_{W_t} \ff{\LL^0(\d w)}{w(g)^\theta}= \ff 1 {\GG(\theta)} \int_0^\infty r^{\theta-1} \e^{-t\nu_0(1-\e^{-rg})}\,\d r.$$ Combining this with (\ref{WFY1}) and
 (\ref{WFY2}) and letting $\vv\to 0$, we obtain the desired gradient estimate. According to the proof of Theorem \ref{T1.3} in Section 4 with $P_t^1$ replaced by $P_t$, this along with the assumption on $T_t$   implies the second assertion.\end{proof}

\section{Proof of Theorem \ref{T1.2}}
By the triangle inequality for $\|\cdot\|_{var}$, it suffices to prove both  assertions for small enough $\|y\|_\BB$.
\subsection{Case (i)}
Let $\|y\|_\BB\le \ff {1\land (\ff{r_0}2)} {1+c\|\si^{-1}\|_\BB}$, which implies that
\beq\label{BB} \|\si^{-1}T_s y\|_\BB+\|y\|_\BB \le 1\land \ff{r_0} 2,\ \ \ s\in [0,t].\end{equation} Moreover, since $\|P_t (x,\cdot)-P_t(x+y,\cdot)\|_{var}\le 2$ holds for all $x,y\in\BB$ and $t>0$,  we only have to prove the desired inequality for large $t>0$. From now on, let us assume $t\ge 2$ and (\ref{BB}).

 Now, let $t\ge 2$ and $x,y\in \BB$ such that (\ref{BB}) holds. Since $T_s \si$ is bounded in $\BB$ uniformly in $s$, for any $z\in\BB$,
 $$J^z(w):= T_t z+ \int_0^t T_{t-s}\si\,\d w_s$$ is $\LL$-a.e. (also $\LL^1$-a.e. and $\LL^0$-a.e.) defined. Moreover, due to (\ref{S}) and $L=L^1+L^0$,
 \beq\label{PP}X_t^z= J^z(L)= J^z(L^1+L^0),\ \ \ z\in \BB, t>0.\end{equation}
 Next, let
$$ \tau_1(w)= \inf\{s>0: \DD w_s\ne 0\}, \ \ \tau_{i+1}(w)= \inf\{s>\tau_i(w):\ \DD w_s\ne 0\},\ \ \ i\ge 1.$$ Since
 $\ll_0=\nu_0(\BB)\in (0,\infty),$ we have $\P(\tau_1(L^0)\ge s)=\e^{-\ll_0 s}\in (0,1)$ for $s>0$, and $\tau_i(L^0)\uparrow \infty$
 as $i\uparrow \infty$. Moreover, let $$N_s(w)=\#\{i\ge 1:\ \tau_i(w)\le s\},\ \ \ s\ge 0.$$ Then $\{N_s(L^0)\}_{s\in [0,t]}$ is a Poisson process
 with parameter $\ll_0$.
Similarly, let   $$ \tt \tau_1(w)= \inf\{s> 1: \DD w_s\ne 0\}, \ \ \tt \tau_{i+1}(w)= \inf\{s>\tt \tau_i(w):\ \DD w_s\ne 0\},\ \ \ i\ge 1$$ and
 $$\tt N_s(w)= N_{s+1}(w)-N_1(w)=\#\{i\ge 1:\ \tt\tau_i\le s+1\}=\#\{i\ge 1: \ 1<\tau_i\le s+1\},\ \ s\in [0,t-1].$$ Then $\{\tt N_s(L^0)\}_{s\in [0,t-1]}$ is a Poisson process with parameter $\ll_0$, which is independent of $\{\tau_1(L^0)> \vv\}=\{N_\vv(L^0)=0\}$ for $\vv\in (0,1).$
 Finally, let
 \beg{equation*}\beg{split} &\xi_i(w)=  1_{B(z_0,\ff{r_0}2)}(\DD w_{\tt\tau_i(w)}),\\
 &\tt\xi_i(w) = \ff{\rr_0(\DD w_{\tt\tau_i(w)}+ \si^{-1}T_{\tt\tau_i(w)}y)}{\rr_0(\DD w_{\tt\tau_i(w)})}
 \big(1_{B(z_0-\si^{-1}T_{\tt\tau_i(w)}y, \ff{r_0}2)}\vp_{-\si^{-1}T_{\tt\tau_i(w)}y}\big)(\DD w_{\tt\tau_i(w)}),\ \ \ i\ge 1.\end{split}\end{equation*} We have
 \beq\label{2*1} \beg{split} &\int_{\BB\times [1,t]}1_{B(z_0,\ff{r_0}2)(z)}\, w(\d z,\d s)= \sum_{i=1}^{\tt N_{t-1}(w)} \xi_i(w),\\
 &\int_{\BB\times [1,t]} \ff{\rr_0(z+\si^{-1}T_s y)}{\rr_0(z)} \big(1_{B(z_0-\si^{-1}T_s y, \ff{r_0}2)}\vp_{-\si^{-1}T_sy}\big)(z)\,w(\d z,\d s)=
 \sum_{i=1}^{\tt N_{t-1}(w)} \tt \xi_i(w),  \end{split}  \end{equation} where
 we set $\sum_{i=1}^0 =0$ by convention.  From now on, we will simply denote
 $$\tau_i=\tau_i(L^0),\ \tt\tau_i=\tt\tau_i(L^0),\ \xi_i=\xi_i(L^0),\ N_s= N_s(L^0),\ \tt N_s=\tt N_s(L^0).$$
 To characterize the coupling property of the solution, we first prove the following relation formula for $X_t^x$ and $X_t^{x+y}.$

 \beg{lem}\label{L3.1} For any $f\in \B_b(\BB)$ and $\vv\in (0,1)$,
 $$\E\bigg\{f(X_t^x)1_{\{\tau_1>\vv\}}\sum_{i=1}^{\tt N_{t-1}}\xi_i\bigg\}= \E\bigg\{f(X_t^{x+y})1_{\{\tau_1>\vv\}}\sum_{i=1}^{\tt N_{t-1}}\tt\xi_i\bigg\}.$$\end{lem}

 \beg{proof} Since $\vv\in (0,1)$,   $\{\tau_1(w)>\vv\}= \{\tau_1(w+z1_{[s,t]})>\vv\}$ holds for $s\in [1,t]$ and $z\in\mathbb B$. Moreover, by the definition of $J^x$ we have
 $$J^x(w^1+w^0)+ T_{t-s} \si z= J^x(w^1+w^0 +z1_{[s,t]}).$$ By   Lemma \ref{L2.1} for
 $$h(w^0,z,s) = f(J^x(w^1,+w^0)+ T_{t-s}\si z)1_{\{\tau_1\ge \vv\}\times B(z_0,\ff{r_0}2)\times [1,t]}(w^0,z,s) $$ with fixed $w^1$  and using (\ref{2*1}),
 we obtain
 \beg{equation*} \beg{split}&\int_{W_t^2}\LL^1(\d w^1)\,\LL^0(\d w^0)\int_{B(z_0,\ff{r_0}2)\times [1,t]} f(J^x(w^1+w^0)+T_{t-s}\si z) 1_{\{\tau_1>\vv\}}(w^0)\,\nu_0(\d z)\,\d s\\
 &= \int_{W_t^2}\LL^1(\d w^1)\,\LL^0(\d w^0)\int_{B(z_0,\ff{r_0}2)\times [1,t]} f(J^x(w^1+w^0 + z1_{[s,t]})) 1_{\{\tau_1>\vv\}}(w^0+z1_{[s,t]})\,\nu_0(\d z)\,\d s\\
 &= \int_{W_t^2}1_{\{\tau_1>\vv\}}(w^0)f(J^x(w^1+w^0 ) )\, \LL^1(\d w^1)\,\LL^0(\d w^0)\int_{B(z_0,\ff{r_0}2)\times [1,t]} \, w^0(\d z,\d s).\end{split}\end{equation*}
 Combining this with (\ref{PP}) and the first equation in (\ref{2*1}) we arrive at
   \beg{equation} \label{2*2}\beg{split}&\int_{W_t^2}\LL^1(\d w^1)\,\LL^0(\d w^0)\int_{B(z_0,\ff{r_0}2)\times [1,t]} f(J^x(w^1+w^0)+T_{t-s}\si z) 1_{\{\tau_1>\vv\}}(w^0)\,\nu_0(\d z)\,\d s\\
 &= \E\bigg\{f(X_t^x)1_{\{\tau_1>\vv\}}\sum_{i=1}^{\tt N_{t-1}}\xi_i\bigg\}.\end{split}\end{equation}

 On the other hand, noting that
 $$J^x(w^1+w^0)+ T_{t-s}\si z= J^{x+y} (w^1+w^0+(z-\si^{-1}T_s y)1_{[s,t]}),$$ by Lemma \ref{L2.1} and the integral transform
 $z\mapsto z+\si^{-1}T_s y,$ we obtain
 \beg{equation*}\beg{split} &\int_{W_t^2}\LL^1(\d w^1)\,\LL^0(\d w^0)\int_{B(z_0,\ff{r_0}2)\times [1,t]} f(J^x(w^1+w^0)+T_{t-s}\si z) 1_{\{\tau_1>\vv\}}(w^0)\,\nu_0(\d z)\,\d s\\
 &=\int_{W_t^2}\LL^1(\d w^1)\,\LL^0(\d w^0)\int_{B(z_0,\ff{r_0}2)\times [1,t]} f(J^{x+y}(w^1+w^0 +\{z-\si^{-1}T_sy\}1_{[s,t]}))\\
 &\qquad\qquad\qquad\qquad \qquad\qquad\qquad\qquad\times   1_{\{\tau_1>\vv\}}(w^0+
 \{z-\si^{-1}T_sy\}1_{[s,t]})\,\nu_0(\d z)\,\d s\\
 &=\int_{W_t^2}\LL^1(\d w^1)\,\LL^0(\d w^0)\int_{[1,t]}\d s \int_{B(z_0-\si^{-1}T_sy,\ff{r_0}2)} f(J^{x+y}(w^1+w^0 +z1_{[s,t]})) \\
 &\qquad\qquad\qquad\qquad \qquad\qquad\times  1_{\{\tau_1>\vv\}}(w^0+
 z1_{[s,t]})\ff{\rr_0(z+\si^{-1}T_s y)}{\rr_0(z)} \vp_{-\si^{-1}T_s y}(z)\,\nu_0(\d z)\\
 &=\int_{W_t^2}1_{\{\tau_1>\vv\}}(w^0)f(J^{x+y}(w^1+w^0 ) ) \,\LL^1(\d w^1)\,\LL^0(\d w^0)
 \int_{\BB\times [1,t]}\ff{\rr_0(z+\si^{-1}T_sy)}{\rr_0(z)}\\
 &\qquad\qquad\qquad\qquad \qquad\qquad\times (1_{B(z_0-\si^{-1}T_s y,\ff{r_0}2)}\vp_{-\si^{-1}T_sy})(z)\, w^0(\d z,\d s).\end{split}\end{equation*}
 Combining this with (\ref{PP}) and the second equation in (\ref{2*1}), we conclude that
  \beg{equation*}  \beg{split}&\int_{W_t^2}\LL^1(\d w^1)\,\LL^0(\d w^0)\int_{B(z_0,\ff{r_0}2)\times [1,t]} f(J^x(w^1+w^0)+T_{t-s}\si z) 1_{\{\tau_1>\vv\}}(w^0)\,\nu_0(\d z)\,\d s\\
 &= \E\bigg\{f(X_t^{x+y})1_{\{\tau_1>\vv\}}\sum_{i=1}^{\tt N_{t-1}}\tt\xi_i\bigg\}.\end{split}\end{equation*}
 The desired formula follows from  this and  (\ref{2*2}).
 \end{proof}

 \beg{lem}\label{L3.2}  Given $\tt N$, $\{\xi_i\}$ and $\{\tt \xi_i\}$ are two conditionally i.i.d. sequences with
 $$\E(\xi_i|\tt N)=    \E(\xi_i^2|\tt N)= \ff{\nu_0(B(z_0,\ff{r_0}2))}{\ll_0},$$ and
 $$\E(\tt\xi_i|\tt N)=\ff{\nu_0(B(z_0,\ff{r_0}2))}{\ll_0},\ \ \ \E(\tt\xi_i^2|\tt N)\le \ff{ \dd_1(\tt\tau_1)}{\ll_0},\ \ i\ge1.$$\end{lem}

 \beg{proof}  Since $\{\DD L^0_{\tt\tau_i}\}$ are i.i.d. and   independent of $\tt N$ with common  distribution $\ff 1 {\ll_0} \nu_0$,
 and since $\tt \tau_i$ is determined by $\tt N$,    it is clear that both $\{\xi_i\}$ and $\{\tt\xi_i\}$ are conditionally i.i.d. sequences given $\tt N$.
 Moreover, we have $$\E(\xi_i^2|\tt N)=\E(\xi_i|\tt N)=\E \xi_i = \ff{\nu_0(B(z_0,\ff{r_0}2))}{\ll_0}. $$ Noting that $\nu_0(\d z)=\rr_0(z)\,\mu(\d z)$ and
 $\mu(\d z+h)=\vp_{-h}(z)\,\mu(\d z)$, we have
\beg{equation*}\beg{split} \E(\tt \xi_i|\tt N)
&= \ff 1 {\ll_0} \int_{B(z_0-\si^{-1}T_{\tt\tau_i} y,\ff{r_0}2)} \ff{\rr_0(z+\si^{-1}T_{\tt\tau_i}y)}{\rr_0(z)}\vp_{-\si^{-1}T_{\tt\tau_i}y}(z)\,\nu_0(\d z)\\
&=\ff 1 {\ll_0} \int_{B(z_0-\si^{-1}T_{\tt\tau_i}y,\ff{r_0}2)} \rr_0(z+\si^{-1}T_{\tt\tau_i}y) \vp_{-\si^{-1}T_{\tt\tau_i}y}(z)\,\mu(\d z)\\
&=\ff 1 {\ll_0}  \int_{B(z_0,\ff{r_0}2)}  \rr_0(z)\,  \mu(\d z)\\
&=\ff{\nu_0(B(z_0,\ff{r_0}2))}{\ll_0}.\end{split}\end{equation*}  Moreover, since $\|\si^{-1}T_{\tt\tau_i}y\|_\BB\le 1\land \ff{r_0}2$   and $\tt\tau_i\ge\tt\tau_1$, we obtain
\beg{equation*}\beg{split} \E({\tt \xi_i}^2|\tt N)
&= \ff 1 {\ll_0} \int_{B(z_0-\si^{-1}T_{\tt\tau_i}y,\ff{r_0}2)} \ff{\rr_0(z+\si^{-1}T_{\tt\tau_i}y)^2}{\rr_0(z)^2}\vp_{-\si^{-1}T_{\tt\tau_i}y}(z)^2\,\nu_0(\d z)\\
&\le \ff 1 {\ll_0} \int_{B(z_0,r_0)}  \ff{\rr_0(z+\si^{-1}T_{\tt\tau_i} y)^2 \vp_{-\si^{-1}T_{\tt\tau_i}y}(z)^2}{\rr_0(z)}\,\mu(\d z)\\
&\le \ff{\dd_1(\tt\tau_1)}{\ll_0}.\end{split}\end{equation*}This completes the proof.
 \end{proof}

\beg{proof}[Proof of Theorem \ref{T1.2} (i)]   As explained in the beginning of this section, we assume that $t\ge 2$ and let $y$ satisfy (\ref{BB}).
By Lemma \ref{L3.1} and $\tt\tau_1\ge\tau_1$, for any $f\in \B_b(\BB)$ with $\|f\|_\infty\le 1$ we have
\beq\label{3.3} \beg{split} & \Big|\E\big(f(X_t^x)-f(X_t^{x+y})\big)1_{\{\tau_1>\vv\}}\Big|\\
&\le \E\bigg|1-\ff 1 {\nu_0(B(z_0,\ff{r_0}2))(t-1)}\sum_{i=1}^{\tt N_{t-1}}\xi_i\bigg| + \E \bigg\{1_{\{\tt\tau_1>\vv\}} \bigg|1-\ff 1 {\nu_0(B(z_0,\ff{r_0}2))(t-1)}\sum_{i=1}^{\tt N_{t-1}}\tt\xi_i\bigg| \bigg\}.\end{split}\end{equation}
Noting    that $\tt\tau_1$ is determined by $\tt N$, we obtain from Lemma \ref{L3.2} that
\beg{equation*}\beg{split} & \E \bigg\{1_{\{\tt\tau_1>\vv\}} \bigg|1-\ff 1 {\nu_0(B(z_0,\ff{r_0}2))(t-1)}\sum_{i=1}^{\tt N_{t-1}}\tt\xi_i\bigg| \bigg\}^2\\
&=\E\bigg\{1_{\{\tt\tau_1>\vv\}}\bigg(\ff {\sum_{i,j=1}^{\tt N_{t-1}} \E(\tt \xi_i\tt \xi_j|\tt N)} {\nu_0(B(z_0,\ff{r_0}2))^2(t-1)^2}-
\ff {2\sum_{i=1}^{\tt N_{t-1}} \E(\tt\xi_i|\tt N)} {\nu_0(B(z_0,\ff{r_0}2))(t-1)}
+1\bigg)\bigg\}\\
&\le \E\bigg\{1_{\{\tt\tau_1>\vv\}}\bigg(\ff {\tt N_{t-1}^2-\tt N_{t-1}}{\ll_0^2(t-1)^2}+ \ff{\tt N_{t-1}\dd_1(\tt\tau_1)}{\ll_0\nu_0(B(z_0,\ff{r_0}2))^2(t-1)^2}
  -\ff {2\tt N_{t-1}} {\ll_0(t-1)} +1\bigg)\bigg\}\\
&\le  \ff{\dd_1(\vv)}{\nu_0(B(z_0,\ff{r_0}2))^2(t-1)}.\end{split}\end{equation*}
Similarly and even simpler, we have
$$\E\bigg(1-\ff 1 {\nu_0(B(z_0,\ff{r_0}2))(t-1)}\sum_{i=1}^{\tt N_{t-1}}\xi_i\bigg)^2\le \ff{1}{\nu_0(B(z_0,\ff{r_0}2))(t-1)}.$$ Combining these with (\ref{3.3}) and noting that
$t-1\ge 1$, we arrive at
$$\Big|\E\big(f(X_t^x)-f(X_t^{x+y})\big)1_{\{\tau_1>\vv\}}\Big|\le \ff{C_1\ss{\dd_1(\vv)}}{\ss t}$$ for some constant $C_1>0$ independent of $t,x,y$ and $\vv\in (0,1)$.
Therefore, there exists a constant $C>0$ independent of $t,x,y$ and $\vv\in (0,1)$ such that for $\|f\|_\infty\le 1$,
\beg{equation*}\beg{split} |P_t f(x)-P_t f(x+y)|&\le \ff{C_1\ss{\dd(\vv)}}{\ss t} +\E\Big|\big(f(X_t^x)-f(X_t^{x+y})\big)1_{\{\tau_1\le\vv\}}\Big|\\
&\le  \ff{C_1\ss{\dd_1(\vv)}}{\ss t}  + 2\P(\tau_1\le\vv)\\
&=  \ff{C_1\ss{\dd_1(\vv)}}{\ss t}+2(1-\e^{-\ll_0\vv})\\
&\le C\bigg(\vv+\ff{\ss{\dd_1(\vv)}}{\ss t}\bigg).\end{split}\end{equation*} This completes the proof.
\end{proof}

\subsection{Case (ii)}
For every
$\eta>0$, define ${\nu}_\eta$ on $\BB$ as follows:
\begin{equation*}
    {\nu}_\eta(A)
    =
    \begin{cases}
        \nu(A),                           & \text{if\ \ } \nu(\BB)<\infty;\\
        \nu(A\setminus \{z: \|z\|_{\BB}<\eta\}), & \text{if\ \ } \nu(\BB)=\infty,
    \end{cases}
\end{equation*} where $A\in \B$. Then $\nu_\eta$ is a finite measure on $(\BB,\B)$.
Recall that for any two finite measures $\pi_1$ and $\pi_2$ on $(\BB,\mathscr{B})$,
$\pi_1\wedge\pi_2:=\pi_1-(\pi_1-\pi_2)^+$, where $(\pi_1-\pi_2)^{\pm}$ refers to the
Jordan-Hahn decomposition of the signed measure $\pi_1-\pi_2$. In
particular, $\pi_1\wedge\pi_2=\pi_2\wedge\pi_1$, and $$(\pi_1\wedge
\pi_2)(\BB)=\frac{1}{2}\big(\pi_1(\BB)+\pi_2(\BB)-\|\pi_1-\pi_2\|_{var}\big).$$
The following is an extension of the main result in \cite{SW2} to the infinite-dimensional setting.

\begin{thm}\label{th1}
  Let $X_t$ be the process determined by \eqref{E1}. Assume that $\sigma$ is invertible, and that there exist $\eta,\varrho>0$ such that
\begin{equation}\label{th2233}
    \gamma(\eta,  \varrho,\vv):=\inf_{t\ge\vv, \|x\|_{\BB}\le \varrho}\big\{\nu _\eta\wedge (\delta_{\sigma^{-1}T_tx}*\nu_\eta)\big\}(\BB)>0
\end{equation} holds for any $\vv>0$.
      Then there exists a   constant  $C >0 $    such that for all $x,y\in\BB$ and $t>0$,
\begin{equation}\label{th21} \|P_t(x,\cdot)-P_t(x+y,\cdot)\|_{var}  \le
  C \big(1+\|y\|_{\BB}\big)\inf_{\vv\in (0,1)} \bigg(\vv+   \frac{1}{\sqrt{\gamma(\eta,  \varrho,\vv)t}} \bigg).
\end{equation}
\end{thm}

We postpone the proof to the end of this subsection and present the proof of Theorem \ref{T1.2} (ii).
 \beg{proof}[Proof of Theorem \ref{T1.2} (ii)]  Without loss of generality, we assume that  $0\notin B(z_0,  r_0 )$. Otherwise, we may take $z_0'\in  B(z_0,r_0)$ and
$r_0'>0$ such that $0\notin B(z_0',r_0')\subset B(z_0,r_0)$, and use $B(z_0',r_0')$ to replace $B(z_0,r_0).$ Moreover, we   take $\varrho\in(0,1)$ small enough  such  that $\|\sigma^{-1}T_t x\|\le 1\land \ff{r_0}4$ holds for all $\|x\|_{\BB}\le \varrho$ and $t>0$.

By \eqref{C}, \eqref{1.1} and the Cauchy-Schwarz inequality, for any $t\ge \vv$ and $\eta\in (0,\ff{r_0}4)$,
\begin{equation*}\label{th211}\begin{aligned}&\inf_{t\ge \vv, \|x\|_{\BB}\le \varrho} \big\{\nu_\eta\wedge (\delta_{\sigma^{-1}T_t x }*\nu_\eta)\big\}(\BB)\\
&\ge \inf_{t\ge \vv, \|x\|_{\BB}\le \varrho}\int_{B(x_0,\ff{r_0}2)}\Big(\rho_0(z)\wedge \big(\rho_0(z-{\sigma^{-1}T_t x })\varphi_{\sigma^{-1}T_t x }(z)\big)\Big)\, \mu(\d z)\\
&\ge {\inf_{t\ge \vv, \|x\|_{\BB}\le \varrho} \Big(\int_{B(x_0,\ff{r_0}2)} \varphi_{\sigma^{-1}T_t x }(z)\,\mu(\d z)\Big)^2}\\
&\qquad\qquad \times\bigg[{\sup_{t\ge \vv, \|x\|_{\BB}\le \varrho}\int_{B(x_0,\ff{r_0}2)}\frac{\varphi_{\sigma^{-1}T_t x }(z)^2}{\rho_0(z)\wedge \big(\rho_0(z-{\sigma^{-1}T_t x })\varphi_{\sigma^{-1}T_t x }(z)\big)}\,\mu(\d z)}\bigg]^{-1}.
\end{aligned}\end{equation*}

Since the measure $\mu$ has full support, \begin{equation*}\label{th211}\begin{aligned}&\inf_{t\ge \vv, \|x\|_{\BB}\le \varrho}\int_{B(x_0,\ff{r_0}2)} \varphi_{\sigma^{-1}T_t x }(z)\,\mu(\d z)\\
&=\inf_{t\ge \vv, \|x\|_{\BB}\le \varrho}\int_{B(x_0,\ff{r_0}2)}\,\mu(\d z- \sigma^{-1}T_t x)\\
&\ge \int_{B(x_0, \ff{r_0}4)} \,\mu(\d z)>0.  \end{aligned}\end{equation*}

On the other hand, by \eqref{C22}, for any $t\ge \varepsilon$,
\begin{equation*}\label{th211}\begin{aligned}&{\sup_{t\ge \vv, \|x\|_{\BB}\le \varrho}\int_{B(x_0,\ff{r_0}2)}\frac{\varphi_{\sigma^{-1}T_t x }(z)^2}{\rho_0(z)\wedge \big(\rho_0(z-{\sigma^{-1}T_t x })\varphi_{\sigma^{-1}T_t x }(z)\big)}\,\mu(\d z)} \\
&\le \sup_{t\ge \vv, \|x\|_{\BB}\le \varrho}\bigg[\int_{B(x_0,\ff{r_0}2)}\frac{\varphi_{\sigma^{-1}T_t x }(z)^2}{\rho_0(z)}\,\mu(\d z)+\int_{B(x_0,\ff{r_0}2)}\frac{\varphi_{\sigma^{-1}T_t x }(z)}{\rho_0(z-\sigma^{-1}T_t x)}\,\mu(\d z)\\
&\le \sup_{t\ge \vv, \|x\|_{\BB}\le \varrho}\bigg[\int_{B(x_0,\ff{r_0}2)}\frac{\varphi_{\sigma^{-1}T_t x }(z)^2}{\rho_0(z)}\,\mu(\d z)+\int_{B(x_0,\ff{r_0}2)}\frac{\mu(\d z-\sigma^{-1}T_t x )}{\rho_0(z-\sigma^{-1}T_t x )}\bigg]\\
&\le\sup_{t\ge \vv, \|x\|_{\BB}\le \varrho}\bigg[\int_{B(x_0,\ff{r_0}2)}\frac{\varphi_{\sigma^{-1}T_t x }(z)^2}{\rho_0(z)}\,\mu(\d z)+\int_{B(x_0,r_0)}\frac{1}{\rho_0(z)}\,\mu(\d z)\bigg]\\
&<\infty.\end{aligned}\end{equation*}
The required assertion \eqref{Coupling1} follows from the conclusions above and \eqref{th21}.
\end{proof}

\begin{proof}[Proof of Theorem \ref{th1}] As indicated in the proof of Theorem \ref{T1.2} (i), we only have to prove the result for  $\|x-y\|_\BB\le\varrho$ and $t\ge 1$.
To this end, we modify the argument from  the proof of \cite[Theorem 1.1]{SW2}. For any $\eta>0$, let $L^\eta$ be a compound Poisson process on $\BB$ with L\'{e}vy measure $\nu_\eta$ such that  $L^\eta$ and $L-L^\eta$ are independent L\'{e}vy processes. Then the random variables
$$
    X_t^{\eta,x}:=T_tx+\int_0^t T_{t-s}\sigma\,\d L_s^\eta
$$
and
$$
    X_t^x-X_t^{\eta,x}:=\int_0^t  T_{t-s}\sigma\,d(L_s-L_s^\eta)
$$
are independent.
Denote by $\mu_{\eta,t}$ the law of random variable
$$
    X_t^{\eta,0} := X_t^{\eta,x}-T_tx = \int_0^tT_{t-s}\sigma\,\d L_s^\eta.
$$
 Construct a sequence $\{\tau_i\}$ of i.i.d.\ random variables which are exponentially distributed with intensity $C_\eta=\nu_\eta(\BB)$, and introduce a further sequence $\{U_i\}$ of i.i.d.\ random variables on $\BB$ with law $\bar{\nu}_\eta=\nu_\eta/C_\eta$. We will assume that the random variables $\{U_i\}$ are independent of the sequence $\{\tau_i\}$. Then, according to \cite[Examples, Section 2]{App1}, $
L_t^\eta=\sum_{i=1}^{N_t}U_i$ for every $t\ge0$, where $N_t:=\sup\{k: \sum_{i=1}^k\tau_i\le t\}$, for $\sum_{i\in \varnothing}:=0$ by convention,  is a Poisson process of intensity $C_\eta$. Therefore, the random variable
\begin{equation}\label{proofs0}
     1_{\{\tau_1\le t\}}\sum_{k=1}^\infty1_{\{N_t=k\}} \Big(T_{t-\tau_1}\sigma U_1+\cdots+T_{t-(\tau_1+\cdots+\tau_k)}\sigma U_k\Big)
\end{equation}
has the probability distribution $\mu_{\eta,t}$.

Let $P_t(x,\cdot)$ and $P_t$ be the
transition kernel and the transition semigroup of the
Ornstein-Uhlenbeck process $X^x_t$. Similarly, we denote
by  $P^\eta_t(x,\cdot)$ and $P^\eta_t$ the transition
kernel and the transition semigroup of
$X_t^{\eta,x}$, and by $Q^\eta_t(x,\cdot)$
and $Q^\eta_t$ the transition kernel and the transition
semigroup of $X_t^x-X_t^{\eta,x}$. By the
independence of the processes $X_t^{\eta,x}$ and
$X_t^x-X_t^{\eta,x}$, we get
\begin{equation}\label{proofs2}\begin{aligned}
    \|P_t(x,\cdot)-P_t(y,\cdot)\|_{var}
    &= \sup_{\|f\|_\infty\le 1}\big|P_tf(x)-P_tf(y)\big|\\
    &=\sup_{\|f\|_\infty\le 1} \big|P_t^\eta Q_t^\eta f(x)-P_t^\eta Q_t^\eta f(y)\big|\\
    &\le \sup_{\|h\|_\infty\le 1} \big|P^\eta_th(x)-P^\eta_th(y)\big|\\
    &=\sup_{\|h\|_\infty\le 1}\Big|\E(h(X_t^{\eta,x}))-\E(h(X_t^{\eta,y}))\Big|.
\end{aligned}\end{equation}

Following the argument leading to \cite[(2.11)]{SW2}, we may write
$$
    \E f\bigl(X_t^{\eta,x}\bigr)
    =\int_\BB f\bigl(T_tx+z\bigr)\,\mu_{\eta,t}(\d z)
    =f\bigl(T_tx\bigr)\,\e^{-C_\eta t}+Hf(x),\ \ f\in\B_b(\BB)
$$
for
$$
Hf(x) =\sum_{k=1}^\infty   \int_{I_{t,k}}  C_\eta^{k+1}\e^{-C_\eta(t_1+\cdots+t_{k+1})}\,\d t_1\cdots \d t_{k+1}\int_{\BB} f\bigl(T_tx+z\bigr)\,\mu_{t_1,\cdots,t_k}(\d z),$$ where \beg{equation*}\beg{split} &I_{t,k}:= \Big\{(t_1,\cdots, t_k,t_{k+1})\in (0,\infty)^{k+1}: \ \sum_{i=1}^k t_i\le t<\sum_{i=1}^{k+1}t_i\Big\},\\
& \mu_{t_1,\cdots,t_{k}} := (\bar \nu_\eta)^k\circ J_{t_1,\cdots, t_k}^{-1},\\
&J_{t_1,\ldots,t_k}(y_1,\ldots,y_k)
    :=T_{t-t_1}\sigma y_1+\cdots+T_{t-(t_1+\cdots+t_k)}\sigma y_k, \ \ y_1,\cdots, y_k\in\BB.\end{split}\end{equation*}
Then, for any $t\ge1$ and $\vv\in (0,1)$,
\begin{equation}\label{proofs3}\begin{aligned}
   &\sup_{\|h\|_\infty\le 1}\Big|\E(h(X_t^{\eta,x}))-\E(h(X_t^{\eta,y}))\Big|\\
    &\le \sup_{\|h\|_\infty\le 1}\bigg|\E\Big(\!\big(h(X_t^{\eta,x})-h(X_t^{\eta,y})\big)1_{\{\tau_1\le \vv\}}\!\Big)\bigg|
   + \sup_{\|h\|_\infty\le 1}\bigg|\E\Big(\!\big(h(X_t^{\eta,x})-h(X_t^{\eta,y})\big)1_{\{\tau_1\ge \vv\}}\!\Big)\bigg|\\
        &\le 2\P( \tau_1\le \vv)+2\e^{-C_\eta t} +\sum_{k=1}^\infty \int_{I_{t,k}\cap\{(0,\infty)^{k+1}:\,t_1\ge\vv\}}C_\eta^{k+1}\e^{-C_\eta(t_1+\cdots+t_{k+1})}\,\d t_1\cdots \d t_{k+1}\\
    &\qquad\qquad\times\sup_{\|h\|_\infty\le 1} \bigg|\int_{\BB} h\big(T_tx+z\big)\,\mu_{t_1,\cdots,t_k}(\d z) - \int_{\BB} h\big(T_ty+z\big)\,\mu_{t_1,\cdots,t_k}(\d z)\biggr|\\
    &=2(1-\e^{-C_\eta\vv})+2\e^{-C_\eta t} +\sum_{k=1}^\infty \int_{I_{t,k}\cap\{(0,\infty)^{k+1}:\,t_1\ge\vv\}} C_\eta^{k+1}\e^{-C_\eta(t_1+\cdots+t_{k+1})}\,\d t_1\cdots \d t_{k+1}\\
    &\qquad\qquad\times\sup_{\|h\|_\infty\le 1} \bigg|\int_{\BB} h\big(T_t(x-y)+z\big)\,\mu_{t_1,\cdots,t_k}(\d z) - \int_{\BB} h(z)\,\mu_{t_1,\cdots,t_k}(\d z)\bigg|\\
    &\le 2C_\eta\vv+ 2\e^{-C_\eta t} \\
    &\quad+\sum_{k=1}^\infty\int_{I_{t,k}\cap\{(0,\infty)^{k+1}:\,t_1\ge\vv\}} C_\eta^{k+1}\e^{-C_\eta(t_1+\cdots+t_{k+1})}\\
    &\qquad\qquad\qquad\qquad\qquad\qquad\quad\times\|\delta_{T_t(x-y)}*\mu_{t_1,\cdots,t_k}-\mu_{t_1,\cdots,t_k}\|_{var}\,\d t_1\cdots \d t_{k+1}.
\end{aligned}\end{equation}

To estimate $\|\delta_{T_t(x-y)}*\mu_{t_1,\cdots,t_k}-\mu_{t_1,\cdots,t_k}\|_{var}$ for any $t_1\ge \vv$ and $ t\ge t_1+\cdots+t_k$, we will use the Mineka and Lindvall-Rogers couplings for random walks as in \cite{SW, SW2}. The remainder of this part is based on steps 4 and 5 in the proof of \cite[Theorem 1.1]{SW2}.
In order to ease notations, we set $\nnu:=\bar{\nu}_\eta$ and $\nnu^{a}:=\delta_a*\bar\nu_\eta$ for any $a\in\BB$.
For any $i\ge 1$, let $(U_i,\Delta U_i)\in \BB \times \BB$ be a pair of random variables with the following distribution
$$
    \Pp\big((U_i,\Delta U_i)\in C\times D\big)
    =
    \begin{cases}
    \qquad \frac 12 (\nnu\wedge\nnu^{-a_i})(C), & \text{if\ \ } D=\{a_i\};\\
    \qquad \frac 12 (\nnu\wedge\nnu^{a_i})(C),  & \text{if\ \ } D= \{-a_i\};\\
    \big(\nnu- \frac 12 (\nnu\wedge\nnu^{-a_i}+\nnu\wedge\nnu^{a_i})\big)(C), & \text{if\ \ } D=\{0\};
    \end{cases}
$$
where $C\in\B$, $a_i=\sigma^{-1}\, T_{t_1+\cdots+t_i}\,(x-y)$ and $D$ is any of the following three sets: $\{-a_i\}$, $\{0\}$ or $\{a_i\}$.
It follows that, cf.\ see \cite[Lemma 3.2]{SW},
\begin{align*}
    \Pp\big(\Delta U_i=-a_i\big)
    =\frac{1}{2}\big(\nnu\wedge\big(\delta_{a_i}*\nnu)\big)(\BB)
    =\frac{1}{2}\big(\nnu\wedge\big(\delta_{-a_i}*\nnu)\big)(\BB)
   =\Pp(\Delta U_i=a_i).
\end{align*}
It is clear that the distribution of $U_i$ is $\nnu$. Let $U_i'=U_i+\Delta U_i$. We claim that the distribution of $U_i'$ is
also $\nnu$. Indeed, for any $C\in\mathscr{B}$,
\begin{align*}
    &\Pp(U_i'\in C)\\
    &=\Pp(U_i-a_i\in C, \Delta U_i=-a_i)
        + \Pp(U_i+a_i\in C, \Delta U_i=a_i)
        +\Pp(U_i \in A, \Delta U_i=0)\\
    &= \frac 12\left(\delta_{-a_i}*(\nnu\wedge\nnu^{a_i})\right)(C)
        +\! \frac12\left(\delta_{a_i}*(\nnu\wedge\nnu^{-a_i})\right)(C)
       \! +\! \left(\!\!\nnu-\!\! \frac 12\,\big(\nnu\wedge\nnu^{-a_i}+\nnu\wedge\nnu^{a_i}\big)\!\!\right)(C)\\
    &=\nnu(C),
\end{align*}
where we have used that
$$
    \delta_{a_i}*(\nnu\wedge\nnu^{-a_i})=\nnu\wedge\nnu^{a_i}\quad\textrm{ and }\quad\delta_{-a_i}*(\nnu\wedge\nnu^{a_i})=\nnu\wedge\nnu^{- a_i}.
$$
Without loss of generality, we can assume that the pairs $(U_i,U_i')$ are independent for all $i\ge1$.  Now we construct the
coupling
$$
    (S_k,S_k')_{k\ge1}
    =\left(\sum_{i=1}^k T_{t-(t_1+\cdots+t_i)}\sigma U_i\big),
        \sum_{i=1}^kT_{t-(t_1+\cdots+t_i)}\sigma U'_i\big)\right)_{k\ge1}
$$
of
$$
    (S_k)_{k\ge1}:=\bigg(\sum_{i=1}^kT_{t-(t_1+\cdots+t_i)}\sigma U_i\bigg)_{k\ge1}.
$$
Since $U'_i-U_i=\Delta U_i$ is either $\pm a_i$ or $0$, we know that
\begin{align*}
(S_k'- S_k)_{k\ge1}=\left(\sum_{i=1}^kT_{t-(t_1+\cdots+t_i)}\sigma(U'_i-U_i)\big)\right)_{k\ge1}=\left(\sum_{i=1}^kT_{t-(t_1+\cdots+t_i)} \sigma\Delta U_i\big)\right)_{k\ge1}
\end{align*}
is a random walk on $\BB$ whose steps are independent and attain the values $-T_t(x-y)$, $0$ and $T_t(x-y)$ with probabilities $\frac 12(1-p_i)$, $p_i$ and $\frac 12(1-p_i)$, respectively; the values of the $p_i$ are given by
\begin{align*}
    p_i
    &= \left(\nnu- \tfrac 12 (\nnu\wedge\nnu^{-a_i}+\nnu\wedge\nnu^{a_i})\right)(\BB)
    = 1-\nnu\wedge\nnu^{-a_i}(\BB).
\end{align*}
Note that
$\mu_{t_1,\cdots,t_k}$ is the law of the random variable
$
    \sum_{i=1}^kT_{t-(t_1+\cdots+t_i)}\sigma U_i.
$
We get
\begin{equation}\label{proofs6}
  \|\delta_{T_t(x-y)}*\mu_{t_1,\cdots,t_k}-\mu_{t_1,\cdots,t_k}\|_{var}
    \le 2\,\Pp(T^S>k),
\end{equation}
where
$$
    T^S=\inf\{i\ge1\::\: S_{i}=S^{\prime}_{i}+T_t(x-y)\}.
$$

From \eqref{th2233} we get that for all $i\ge1$, $t_1\ge \varepsilon$, $t\ge t_1+\cdots+t_k$ and $x, y\in\BB$ with $\|x-y\|_{\BB}\le \varrho$,
\begin{equation}\label{proofcon}\begin{aligned}
    \frac 12(1-p_i)
    &= \frac{1}{2}\big(\nnu\wedge\big(\delta_{-a_i}*\nnu)\big)(\BB)\\
    &\ge \frac{1}{2}\inf_{s\ge \varepsilon, \|z\|_{\BB}\le \varrho}\nnu\wedge (\delta_{\sigma^{-1}T_sz}*\nnu)(\BB)\\
    &=\frac{1}{2C_\eta}\,\gamma(\eta,\varrho,\varepsilon)>0.
\end{aligned}\end{equation}
We will now estimate $\Pp(T^S>k)$. Let $V_i$, $i\geq 1$, be independent symmetric random variables on $\BB$, whose distributions are given by
$$
    \Pp(V_i=z)
    = \begin{cases}
        \frac 12(1-p_i),     &\text{if\ \ } z=-T_t(x-y);\\
        \frac 12(1-p_i),     &\text{if\ \ } z=T_t(x-y);\\
        \qquad p_i,          &\text{if\ \ } z=0.
    \end{cases}
$$
Set $Z_k:=\sum_{i=1}^k V_i$. We have seen earlier that
$$
    T^S=\inf\{k\ge 1\::\: Z_k=T_t(x-y)\}.
$$
For any $k\ge1$, let
$$
    \kappa=\kappa(k):=\#\big\{i\::\: i\le k\textrm{ and }V_i\neq 0\big\}
$$
and set $\tilde{Z}_k :=\sum_{i=1}^k\tilde{V}_i$, where $\tilde{V}_i$ denotes the $i$th $V_j$ such that $V_j\neq 0$. Then, $\tilde{Z}_k$ is a symmetric random walk on $\BB$ with iid steps which are either $-T_t(x-y)$ or $T_t(x-y)$ with probability $1/2$. Define
$$
    T^{\tilde{Z}}:=\inf\{k\ge 1\::\: \tilde{Z}_k=T_t(x-y)\}.
$$
By \eqref{proofcon},
\begin{equation}\label{lll1}\begin{aligned}
    \Pp(T^S>k)
    &=\Pp\left(T^S>k,\; \kappa\ge \frac{1}{2C_\eta}\,\gamma(\eta,\varrho,\varepsilon)k\right)
        +\Pp\left(T^S>k,\, \kappa\le \frac{1}{2C_\eta}\, \gamma(\eta,\varrho,\varepsilon)k\right)\\
    &\le\Pp\left(T^{\tilde{Z}}> \frac{1}{2C_\eta}\,\gamma(\eta,\varrho,\varepsilon)k\right)
        +\Pp\bigg(\kappa\le \frac{1}{2}\sum_{i=1}^k(1-p_i)\bigg)\\
    &\le\Pp\left(T^{\tilde{Z}}>\frac{1}{2C_\eta}\,\gamma(\eta,\varrho,\varepsilon)k\right)
        +\Pp\bigg(\Big|\kappa-\sum_{i=1}^k(1-p_i)\Big|\ge\frac{1}{2}\sum_{i=1}^k(1-p_i)\bigg).
\end{aligned}\end{equation}
Note that
$
    \kappa=\kappa(k)=\sum_{i=1}^k\zeta_i,
$
where $\zeta_i = 1_{\{V_i\neq 0\}}$, ${1\le i\le k}$, are independent random variables with
$\Pp(\zeta_i = 0) = p_i$ and $\Pp (\zeta_i=1)=1-p_i$. Chebyshev's inequality shows that
\begin{equation}\label{lll2}\begin{aligned}
    \Pp\bigg(\Big|\kappa-\sum_{i=1}^k(1-p_i)\Big|\ge\frac{1}{2}\sum_{i=1}^k(1-p_i)\bigg)
    &\le \frac{4 var(\kappa)}{\Big(\sum_{i=1}^k(1-p_i)\Big)^2}\\
    &=\frac{4\sum_{i=1}^kp_i(1-p_i)}{\Big(\sum_{i=1}^k(1-p_i)\Big)^2}\\
    &\le \frac{4(1-C_\eta^{-1}\gamma(\eta,\varrho,\varepsilon))\sum_{i=1}^k(1-p_i)}{\Big(\sum_{i=1}^k(1-p_i)\Big)^2}\\
    &\le\frac{4(1-C_\eta^{-1}\gamma(\eta,\varrho,\varepsilon))}{C_\eta^{-1}\gamma(\eta,\varrho,\varepsilon)k}.
\end{aligned}\end{equation}
For the second and the last inequalities we have used \eqref{proofcon}.
On the other hand, by \cite[Lemma 2.3]{SW2},
\begin{align*}
    \Pp\bigg(T^{\tilde{Z}}>\frac{1}{2C_\eta}\gamma(\eta,\varrho,\varepsilon)k\bigg)
    &=\Pp\bigg(\Big\langle\max_{i\le \big[\frac{\gamma(\eta,\varrho,\varepsilon)k}{2C_\eta}\big]}\tilde{Z}_i, \theta^*\Big\rangle< \| T_t(x-y)\|_{\BB}\bigg)\\
    &=\Pp\bigg(\max_{i\le \big[\frac{\gamma(\eta,\varrho,\varepsilon)k}{2C_\eta}\big]}\|\tilde{Z}_i\|_{\BB}<  \|T_t(x-y)\|_{\BB}\bigg)\\
    &\leq 2\,\Pp\left(0\le \Big\|\tilde{Z}_{\big[\frac{\gamma(\eta,\varrho,\varepsilon)k}{2C_\eta}\big]}\Big\|_{\BB} \leq \|T_t(x-y)\|_{\BB}\right),
\end{align*}where in the first equality $\theta^*$ is an element in the dual space $E^*$ of the Banach space $E$ such that the duality $\langle T_t(x,y), \theta^*\rangle=\|T_t(x-y)\|_{\BB}$, and in the second equality $\|\tilde{Z}_i\|_{\BB}=\langle \tilde{Z}_i, \theta^*\rangle$ for $i\ge1$.
From the construction above, we know that $(\|\tilde{Z}_k\|_{\BB})_{k\ge1}$ is a symmetric random walk on $\R$ with iid steps with values $\pm \|T_t(x-y)\|_{\BB}$. Using the central limit theorem we find for sufficiently large values of $k\geq k_0$ and some constant $C_0=C_0(k_0)\ge 1$
\begin{equation}\label{lll3}\begin{aligned}
    &\Pp\left(T^{\tilde{Z}}>\frac{1}{2C_\eta}\,\gamma(\eta,\varrho,\varepsilon)k\right)\\
    &\le 2\,\Pp\left(0\leq \frac{\Big\|\tilde{Z}_{\big[\frac{\gamma(\eta,\varrho,\varepsilon)k}{2C_\eta}\big]}\Big\|_{\BB}}{\|T_t(x-y)\|_{\BB}\sqrt{\big[\frac{\gamma(\eta,\varrho,\varepsilon)k}{2C_\eta}\big]}} \leq {\left[\frac{\gamma(\eta,\varrho,\varepsilon)k}{2C_\eta}\right]}^{-1/2}\right)\\
    &\leq \frac{C_0}{\sqrt{2\pi}} \int_{0}^{  {\left[\frac{\gamma(\eta,\varrho,\varepsilon)k}{2C_\eta}\right]}^{-1/2}} e^{-u^2/2}\,du\\
    &\leq \frac{C_0\sqrt{C_\eta}}{\sqrt{\pi \gamma(\eta,\varrho,\varepsilon)k}}.
\end{aligned}\end{equation}
Combining \eqref{lll1}, \eqref{lll2} and \eqref{lll3} gives for all $x,y\in\BB$ with $\|x-y\|_{\BB} \le\varrho$, $t\ge (t_1+\cdots+t_k)\vee 1$, $t_1\ge \varepsilon$ and $k\geq k_0$ that
$$
    \Pp\big(T^S>k\big)
    \le \frac{C_0\sqrt{C_\eta}}{\sqrt{\pi\gamma(\eta,\varrho,\varepsilon) k}}+\frac{4(1-C_\eta^{-1}\gamma(\eta,\varrho,\varepsilon))}{C_\eta^{-1}\gamma(\eta,\varrho,\varepsilon)k}.
$$
According to the estimate above and \eqref{proofs6}, we can find an integer $k_0$ and a constant $C_1>0$  such that
$$
  \|\delta_{T_t(x-y)}*\mu_{t_1,\cdots,t_k}-\mu_{t_1,\cdots,t_k}\|_{var}
    \le C_1\bigg(\frac{1}{\sqrt {\gamma(\eta,  \varrho,\vv)k}}+\frac{1}{\gamma(\eta,  \varrho,\vv)k}\bigg),\ \ k\ge k_0,\vv\in (0,1),t\ge 1
$$ holds for all $x,y\in\BB$ with $\|x-y\|_{\BB}\le \varrho$ and  $(t_1,\cdots, t_{k+1})\in I_{t,k}\cap\{(0,\infty)^{k+1}:\,t_1\ge\vv\}.$

Combining this with  \eqref{proofs2} and  \eqref{proofs3}, we obtain that for all $x, y\in\BB$ with $\|x-y\|_{\BB}\le \varrho$, $t\ge1$ and $\vv>0$,
 \begin{equation}\begin{aligned}\label{proofs8}
 &\| P_t(x,\cdot)-P_t(y,\cdot)\|_{var}\\
 &\le  2C_\eta\vv+2\e^{-C_\eta t} +2\sum_{k=1}^{k_0} \int_{I_{t,k}} C_\eta^{k+1}\e^{-C_\eta(t_1+\cdots+t_{k+1})}\,\d t_1\cdots \d t_{k+1}\\
 &\quad
  +\frac{C_1}{\sqrt {\gamma(\eta,  \varrho,\vv)}}\sum_{k=1}^\infty\frac{1}{\sqrt{k}} \int_{I_{t,k}}C_\eta^{k+1}\e^{-C_\eta(t_1+\cdots+t_{k+1})}\,\d t_1\cdots \d t_{k+1}\\
   &\quad +\frac{C_1}{ {\gamma(\eta,  \varrho,\vv)}}\sum_{k=1}^\infty\frac{1}{{k}} \int_{I_{t,k}}C_\eta^{k+1}\e^{-C_\eta(t_1+\cdots+t_{k+1})}\,\d t_1\cdots \d t_{k+1}\\
        &\le 2C_\eta\vv+2\e^{-C_\eta t}\bigg(1+ C_\eta\sum_{k=1}^{k_0} \frac{C_\eta^{k}t^k}{k!}\bigg)\\
    &\quad +\frac{C_1C_\eta}{\sqrt {\gamma(\eta,  \varrho,\vv)}} \sum_{k=1}^\infty\frac{C_\eta^k t^k}{\sqrt{k}\,k!}\e^{-C_\eta t}+\frac{C_1C_\eta}{{\gamma(\eta,  \varrho,\vv)}} \sum_{k=1}^\infty\frac{C_\eta^k t^k}{{k}\,k!}\e^{-C_\eta t}\\
        &\le C_2\bigg(\vv+\e^{-\frac{1}{2}C_\eta t}+ \frac{1}{\sqrt{\gamma(\eta,  \varrho,\vv)t}}+ \frac{1}{{\gamma(\eta,  \varrho,\vv)t}}\bigg)
\end{aligned}\end{equation}holds for some constant $C_2>0$ depending only on $C_\eta$ and $C_1.$
  To finish  the proof, let
\beg{equation*} \beg{split} &\dd_t:= \inf_{\vv>0} \Big(\vv+ \ff 1 {\ss{\gg(\eta,\varrho,\vv)t}}\Big),\\
&\vv_t= \sup\Big\{\vv>0:\ \vv^2\gg(\eta,\varrho, \vv)\le\ff 1 t\Big\},\ \ t\ge 1.\end{split}\end{equation*} Then it is easy to see that $\dd_t,\vv_t\downarrow 0$ as $t\uparrow\infty$ and $\vv_t\ge \ff 1 {\ss{C_\eta t}}.$  Moreover,  since $\gg(\eta,\rr,\vv)$ is increasing    in $\vv$,  for any $\vv\in (0,\vv_t)$,
$$\vv +\ff 1 {\ss{\gg(\eta,\rr,\vv)t}} \ge\lim_{\vv'\uparrow \vv_t} \ff 1 {\ss{\gg(\eta,\rr,\vv')t}}\ge \lim_{\vv'\uparrow \vv_t} \vv'=\vv_t.$$ So, $\dd_t\ge \vv_t$.
Therefore, there exists a constant $C_3>0$ such that
\beg{equation*}\beg{split} \inf_{\vv>0}\bigg(\vv+\e^{-\frac{1}{2}C_\eta t}+ \frac{1}{\sqrt{\gamma(\eta,  \varrho,\vv)t}}+ \frac{1}{{\gamma(\eta,  \varrho,\vv)t}}\bigg)
&\le \dd_t +\e^{-\ff 1 2 C_\eta t}+\dd_t^2 \\
&\le \dd_t+ \e^{-\ff 1 {2\vv_t^2}} +\dd_t^2\\
 & \le \dd_t+ \e^{-\ff 1 {2\dd_t^2}} +\dd_t^2\\
 &\le C_3 \dd_t,\qquad   \ t\ge 1.\end{split}\end{equation*}
 Combining this with (\ref{proofs8})   we complete the proof.\end{proof}

\section{Proof of Theorem \ref{T1.3}} Let $L^1, L^0, \LL^1,\LL^0$ be in Section 3.1. In particular, $L^0$ is a compound Poisson process with jump measure $\nu_0$. Then $L^0$ can be formulated as
$$L^0_t=\sum_{i=0}^{N_t} \xi_i, \quad t>0,$$
where $N_t:= \#\{s\in [0,t]:\ \DD L^0_s\ne 0\}$, $\xi_i=\DD L^0_{\tau_i}$ for $\tau_i$ the $i$-th jump time of $L^0$. It is well-known that $N$, $\{\xi_i\}$ are independent, $N$ is the Poisson process with parameter $\ll_0$, and $\{\xi_i\}$ have common distribution $\ff 1 {\ll_0}\nu_0.$
To derive exponential convergence of $P_t$ in the total variational norm, we make use of the decomposition
\begin{equation}\label{D} \begin{split} P_t f(x)&= \E \big(1_{\{N_t=0\}}f(X_t^x)\big) +P_t^1f(x),\\
 P_t^1f(x)&= \E\big(1_{\{N_t\ge 1\}} f(X_t^x)\big),\quad f\in \B_b(\BB), t\ge 0, x\in \BB.\end{split}\end{equation}
Since when $t\to\infty,  \E \big(1_{\{N_t=0\}}f(X_t^x)\big)$ decays exponentially fast, it suffices to prove the exponential convergence of $P_t^1.$ To this end, we first consider the gradient estimate of $P_t^1.$

\beg{prp}\label{PP2} Assume {\bf (A)} and suppose that  $(\ref{Z1})$   and $(\ref{Z2})$ hold. Let
$$\GG_t:= \ff 1 {1-\e^{-\ll_0 t}} \int_0^t\e^{-\ll_0 r}\Big(\sup_{\|z\|_\BB\le 1}\sup_{s\ge r}\|\si^{-1}T_sz\|_\H\Big)\,\d r <\infty,\ \ t>0.$$ Then  
$$\|\nn P_t^1f\|_\infty\le   c \GG_t \|f\|_\infty,\ \ t>0, f\in \B_b(\BB).$$\end{prp}

\beg{proof} The proof is modified from that of \cite[Theorem 3.1]{W10b}. It suffices to prove  
 \beq\label{GGG1} |\nn_{z_0}P_t^1f(x)|\le c \GG_t\|f\|_\infty,\ \ z_0, x\in\BB, \|z_0\|_\B\le 1, f\in \B_b(\BB).\end{equation}  To prove this inequality, we first establish a formula for $P_t^1$ as in
  \cite[(3.8)]{W10b} where $\si=I$ is considered. Recall that for a random variable $(\xi,\tau)$ on $\BB\times [0,t)$ such that the   distribution of $(L^0,\xi,\tau)$ is
 $$g(w,z,s) \LL^0(\d w) \nu_0(\d z)\d s,$$ \cite[Corollary 2.3]{W10a} implies that
 \beq\label{ABC1} \E\big\{(F1_{\{U>0\}})(L^0)\big\} = \E \ff{F1_{\{U>0\}}}{U}(L^0+ \xi 1_{[\tau,t]})\end{equation} holds for positive measurable function $F$  on $W_t$, where
 \beq\label{ABC2} U(w):= \sum_{s\in [0,t): \DD w_s\ne 0} g(w-\DD w_s 1_{[s,t]}, \DD w_s, s).\end{equation} Now, let $(\xi,\tau)$ be independent of $(L^1,L^0)$ with distribution
 $\ff 1 {t\ll_0} 1_{[0,t]}(s) \nu_0(\d z) \d s.$ We have $g(w,z,s)=\ff 1 {t\ll_0} 1_{[0,t]}(s)$, so that
 $$U(L^0+\xi 1_{[\tau,t]}) =\ff{N_t +1}{\ll_0 t} >0.$$ Therefore, letting $Y_t= \int_0^{t}T_{t-s} \si \d L^1_s$ which is independent of $(L^0, \xi,\tau)$, combining (\ref{S}) with (\ref{ABC1}) we obtain
 \beq\label{ABC3} \beg{split}P_t^1f(x+\vv z_0) &=  \E \bigg\{f\bigg(Y_t + T_t(x+\vv z_0) +\int_0^t T_{t-s} \si \d L_s^0\bigg) 1_{\{N_t\ge 1\}}\bigg\}\\
 & =\ll_0 t\, \E\bigg\{\ff{ f\big( Y_t + T_tx +\int_0^t T_{t-s} \si \d \big\{L^0+(\xi+\vv \si^{-1}T_\tau z_0)1_{[\tau,t]}\big\}_s\big)}{N_t+1}\bigg\}.\end{split}\end{equation}
 On the other hand,    it is easy to see from (\ref{1.1}) that the distribution of $(L^0,  \xi+\vv \si^{-1}T_\tau z_0,\tau)$ is
$$\ff{\vp_{\vv \si^{-1}T_sz_0}(z)\rr_0(z-\vv\si^{-1}T_sz_0)1_{[0,t]}}{t\ll_0\rr_0(z)}\,\LL^0(\d w)\, \nu_0(\d z)\, \d s =: g(w,z,s) \,\LL^0(\d w)\, \nu_0(\d z)\, \d s.$$ According to (\ref{ABC2}) we have $\{U(L^0)>0\}= \{N_t\ge 1\}$ and
$$U(L^0) = \ff 1 {\ll_0 t} \sum_{i=1}^{N_t} \ff{\varphi_{\vv \si^{-1}T_{\tau_i} z_0}(\xi_i) \rr_0(\xi_i-\vv\si^{-1}T_{\tau_i} z_0)}{\rr_0(\xi_i)}.$$ So, applying (\ref{ABC1}) to $FU$ in place of $F$, we obtain
$$\E\big\{(FU)(L^0)1_{\{N_t\ge 1\}}\big\} = \E\big\{(F1_{\{U>0\}})(L^0 +(\xi+\vv \si^{-1}T_\tau z_0)1_{[\tau,t]}\big\}.$$ Taking
$n_t(w)= \sum_{s\le t}1_{\{\nn w_s\ne 0\}}$ such that $N_t= n_t(L^0)$,  and letting
$$F(w)= \ff{f\big(Y_t+T_t x+ \int_{\BB\times [0,t]}T_{t-s}\si z \, w(\d z,\d s)\big)}{n_t(w)} 1_{\{n_t(w)\ge 1\}},$$ we arrive at
\beg{equation*}\beg{split} &\ff 1 {\ll_0 t} \E \bigg\{f\bigg(Y_t +T_t x +\int_0^t  T_{t-s}\si \d L_s^0\bigg) \ff{1_{\{N_t\ge 1\}}}{N_t} \sum_{i=1}^{N_t} \ff{\varphi_{\vv\si^{-1} T_{\tau_i} z_0}(\xi_i) \rr_0(\xi_i-\vv\si^{-1}T_{\tau_i} z_0)}{\rr_0(\xi_i)}\bigg\}\\
&= \E\bigg\{\ff{ f\big(( Y_t + T_tx +\int_0^t T_{t-s} \si \d \big\{L^0+(\xi+\vv \si^{-1}T_\tau z_0)1_{[\tau,t]}\big\}_s\big)}{N_t+1}\bigg\}.\end{split}\end{equation*} Combining this with
(\ref{ABC3}) and noting that $X_t^x=Y_t +T_t x +\int_0^t  T_{t-s} \si \d L_s^0$ due to (\ref{S}),  we obtain
$$P_t^1f(x+\vv z_0)= \E \bigg\{f (X_t^x) \ff{1_{\{N_t\ge 1\}}}{N_t} \sum_{i=1}^{N_t} \ff{\varphi_{\vv\si^{-1} T_{\tau_i} z_0}(\xi_i) \rr_0(\xi_i-\vv\si^{-1}T_{\tau_i} z_0)}{\rr_0(\xi_i)}\bigg\}.$$
Therefore,
\beg{equation}\label{GGG2}\beg{split} &\ff{|P_t^1f(x+\vv z_0)-P_t^1 f(x)|}{\vv}\\
 &=\E\bigg\{f(X_t^x)1_{\{N_t\ge 1\}} \ff 1 {N_t}\sum_{i=1}^{N_t} \ff{\vp_{\vv \si^{-1}T_{\tau_i}z_0}(\xi_i)\rr_0(\xi_i-\vv\si^{-1}T_{\tau_i}z_0)-\rr_0(\xi_i)}{\vv\rr_0(\xi_i)}\bigg\}\\
 &\le \ff{\|f\|_\infty}{\varepsilon\ll_0} \E\bigg\{1_{\{N_t\ge 1\}} \ff 1 {N_t}\sum_{i=1}^{N_t}\int_{\BB} \big|
  \vp_{\vv \si^{-1}T_{\tau_i}z_0}(z)\rr_0(z-\vv\si^{-1}T_{\tau_i}z_0)-\rho_0(z)\big|\,\mu(\d z)\bigg\}\end{split}\end{equation} holds for any $\vv>0.$ Moreover,  it follows from
  (\ref{1.1}) and (\ref{Z1}) that
\beg{equation*}\beg{split}   &\int_{\BB} \big|
  \vp_{\vv \si^{-1}T_{\tau_i}z_0}(z)\rr_0(z-\vv\si^{-1}T_{\tau_i}z_0)-\rho_0(z)\big|\,\mu(\d z)\\
  &\le \int_{\BB} \big|
  \rr_0(z-\vv\si^{-1}T_{\tau_i}z_0)-\rr_0(z)\big|\vp_{\vv \si^{-1}T_{\tau_i}z_0}(z)\,\mu(\d z)+ \int_{\BB}\rr_0(z)\big|\vp_{\vv \si^{-1}T_{\tau_i}z_0}(z)-1\big|\,\mu(\d z)\\
  &= \int_{\BB} \big|\rr_0(z)-
  \rr_0(z+\vv\si^{-1}T_{\tau_i}z_0)\big|\,\mu(\d z)+ \int_{\BB}\rr_0(z)\big|\vp_{\vv \si^{-1}T_{\tau_i}z_0}(z)-1\big|\,\mu(\d z)\\
  &\le c\|\vv \si^{-1}T_{\tau_i}z_0\|_\H\le c\vv \sup_{\|z\|_\B\le 1}\sup_{s\ge \tau_1}\|\si^{-1}\tau_sz\|_\H
    \end{split}\end{equation*} holds for small enough $\vv>0$ and some constant $c>0$. Combining this with (\ref{GGG2}) and  using the fact that the conditional distribution of $\tau_1$ under $N_t\ge 1$ is $\ff {\ll_0\e^{-\ll_0s}1_{[0,t]}}{1-\e^{-\ll_0t}} \, \d s$, we obtain
$$\ff{|P_t^1f(x+\vv z_0)-P_t^1 f(x)|}{\vv}\le   c\GG_t\|f\|_\infty
  $$ for small enough $\vv>0.$ Then (\ref{GGG1}) follows by letting $\vv\to 0$.
  \end{proof}

\beg{proof}[Proof of Theorem \ref{T1.3}] By (\ref{D}) and Proposition \ref{PP2} we have
\beq\label{DD2}\beg{split}  |P_t f(x)-P_t f(y)|&\le 2\|f\|_\infty \e^{-\ll_0 t} + |P_t^1f(x)-P_t^1f(y)|\\
&\le 2\|f\|_\infty\e^{-\ll_0t} + c \GG_t \|f\|_\infty\|x-y\|_\BB.\end{split}\end{equation} Since $\|T_s\|_\BB\le c\e^{-\ll s}$, it follows from (\ref{S}) that
$$\|X_t^x-X_t^y\|_\BB\le c\e^{-\ll t}\|x-y\|_\BB,\ \ x,y\in \BB, t\ge 0.$$ Combining this with (\ref{DD2}) and using the Markov property, we arrive at
\beg{equation*}\beg{split} & |P_t f(x)-P_t f(y)|\\
&\le \E |P_s f(X_{t-s}^x)-P_sf(X_{t-s}^y)|\\
&\le 2\|f\|_\infty\e^{-\ll_0s} + c \GG_s \|f\|_\infty\|X_{t-s}^x-X_{t-s}^y\|_\BB \\
&\le c_1\|f\|_\infty(1+\|x-y\|_\BB) \big\{\e^{-\ll_0s} \lor (\GG_s \e^{-\ll (t-s)})\big\},\ \ \ s\in (0,t) \end{split}\end{equation*} for some constant $c_1>0$.
Taking $s= \ff{\ll t}{\ll_0+\ll}$ and using (\ref{Z2}), we prove the desired estimate for $t\ge \ff{\ll_0+\ll}\ll.$ The proof is then finished since the inequality trivially holds for some constant $C>0$ for $t\le  \ff{\ll_0+\ll}\ll.$ \end{proof}

\section{Two specific models}

In the following two examples   we   take the reference measure  $\mu$ to be the Wiener measure on the Brownian path space,   and the Gaussian measure on a separable Hilbert space, respectively.

\subsection{Wiener measure}

Let    $\BB=\{x\in C([0,1];\R^d):\ x_0=0\}$,  and let $\mu$ be  the Wiener measure on $\BB$, i.e. the distribution of the $d$-dimensional Brownian motion $(B_s)_{s\in [0,1]}$. Let  $\H=\{h\in \BB: \int_0^1|\dot h_s|^2\,\d s<\infty\}$ be the Cameron-Martin space. Then $(\BB, \H,\mu)$ is known as the Wiener space (see \cite[Chapter 1]{Ma}).

By the Cameron-Martin theorem (or the Girsanov theorem),  (\ref{1.1}) holds for
\beq\label{1.2}\vp_h(z)= \exp\bigg[\int_0^1\<\dot h_s, \d z_s\>-\ff 1 2 \int_0^1| \dot h_s|^2\,\d s\bigg],\end{equation} where $ \int_0^1\<\dot h_s, \d z_s\>$ is the It\^o stochastic integral
w.r.t. $(z_s)_{s\in [0,1]}$,  which is the Brownian motion under $\mu$.

 Let $(\BB,\H,\mu)$ be the Wiener space specified above, and let $\DD$ be the Laplace operator on $[0,1]$ with Dirichlet boundary
condition at $0$, and with either Dirichlet or Neumann boundary condition at $1$. We call $\DD$ the Dirichlet or the Dirichlet-Neumann Laplacian on $[0,1].$ Let $P_t$ be the semigroup associated with the SDE
$$\d X_t= \DD X_t\d t+\d L_t,$$ where $L_t$ is a L\'evy process on $\BB$ with L\'evy measure $\nu$, and let $\nu_0$ satisfy $(\ref{C})$.

\beg{prp}\label{P4.1}  $(1)$ If $\nu_0(\BB)=\infty$, then $P_t$ is strong Feller for any $t>0.$

$(2)$ If $\nu_0(\BB)<\infty$ and there exist
$z_0\in\BB$ and $r_0>0$ such that $\inf_{B(z_0,r_0)}\rr_0>0$,   then  $$\|P_t(x,\cdot)-P_t(y,\cdot)\|_{var}\le \ff{C(1+\|x-y\|_\BB)}{\log (1+t)},\ \ \ t>0, x,y\in\BB$$ holds for some constant $C>0.$

$(3)$ If $\rr_0$ is Lipschitz continuous and $\ll_0:=\nu_0(\BB)\in (0,\infty)$, then $(\ref{Z3})$ holds for $\ll>0$ the first eigenvalue of $\DD$ on $[0,1]$ under the underlying boundary  condition. \end{prp}

\beg{proof} By the gradient estimate for the (Dirichlet or Dirichlet-Neumann) heat semigroup $T_s$ on the interval $[0,1]$ (see e.g.\ \cite[Section 2.4]{WangBook} and the references therein), there exists a constant
 $c_1>0$ such that
 $$\Big|\ff{\d}{\d r}(T_s y)(r) \Big|\le \ff{c_1\|y\|_\BB}{\ss s},\ \ \ s>0, r\in [0,1], y\in \BB.$$ Then
 \beq\label{GFY} \|T_s y\|_\H\le \ff{c_1\|y\|_\BB}{\ss s},\ \ \ s>0, y\in\BB.\end{equation}  Therefore, {\bf (A)} holds for $\si= I$.  By (\ref{1.2}) and (\ref{GFY}),  for $\mu$-a.e. Bownian path $z$,  we have
 $$\sup_{\|y\|_\BB\le 1} \vp_{T_s y}(z+T_s y)\le  \sup_{\|h\|_\H\le c_1s^{-1/2}}   \e^{ \int_0^1\<\dot h_u, \d  z_u \> -\ff {1}{2}\int_0^1 |\dot h_u|^2\,\d u }
  <\infty.$$ Thus, (\ref{1.5'}) holds and  Theorem \ref{T1.1} implies the first assertion.

 Next, noting that
 $$\int_\BB \e^{ 2\int_0^1\<\dot h_r, \d z_r\>-2\int_0^1|\dot h_r|^2\d r}\,\mu(\d z)=1,\ \ \ h\in\H,$$ we obtain
 \beg{equation*}\beg{split} \int_\BB \vp_{T_s y}(z)^2\,\mu(\d z) &= \int_\BB \e^{2\<\ff{\d }{d r} (T_s y)_r, \d z_r\> -\int_0^1 |\ff{\d}{\d r} (T_s y)_r|^2\d r}\,
 \mu(\d z)\\
 &= \e^{ \int_0^1 |\ff{\d}{\d r} (T_s y)_r|^2\,\d r}\le \e^{ c_1^2\|y\|_\BB^2/s},\ \ \ s>0, y\in \BB.\end{split}\end{equation*} This implies that   $\dd_2(\vv)\le c_2\e^{ c_1^2/\vv}$  for some constant $c_2>0$ and all $\vv\in (0,1)$. Thus, the second assertion follows from Theorem \ref{T1.2}.

 Finally, to prove (3) it suffices to verify (\ref{Z1}) and (\ref{Z2}) in Theorem \ref{T1.3}. Since (\ref{Z2}) follows from (\ref{GFY}), we only have to prove (\ref{Z1}). By the Lipschitz continuity of $\rr_0$, there exist    constants $c_3, c_4>0$ such that
 $$|\rr_0(z)-\rr_0(z+h)|\le c_3\|h\|_\BB\le c_3\|h\|_\H$$ and
 $$\mu(\rr_0^2)\le c_4\E\sup_{s\in [0,1]}\big(1+| B_s|^2\big)<\infty.$$ Moreover,
 $$\mu\big((\vp_h-1)^2\big)=\mu(\vp_h^2)-1 =\E\e^{2\int_0^1\<\dot h_s,\d B_s\>-\|h\|_\H^2}-1=\e^{\|h\|_\H^2}-1\le \e \|h\|_\H^2$$ holds for $\|h\|_\H\le 1.$ Then (\ref{Z1}) holds for some constant $c>0.$
\end{proof}

\subsection{Gaussian measure} Let $\BB$ be a separable Hilbert space with ONB $\{e_k\}_{k\ge 1}$, and $\mu$ the Gaussian measure with trace class covariance operator $Q$ such that $Q e_k=q_k^{-1} e_k, q_k>0$ and $\sum_{k=1}^\infty q_k^{-1}<\infty$ (see \cite[Chapter 2]{DZ}). Coordinating $z\in \BB$ by $(z_k=\<z,e_k\>)_{k\ge 1}$, we have
\beq\label{EE0} \mu(\d z)=\prod_{k=1}^\infty \mu_k(\d z_k),\ \ \mu_k(\d z_k)=\ff{\ss{q_k}}{\ss{2\pi}}\exp\Big[-\ff{q_kz_k^2}{2}\Big]\d z_k,\ k\ge 1.\end{equation}  Next, let $A$ be the self-adjoint operator on $\BB$ with $A e_k=-\ll_k e_k$, $\ll_k\ge 0$ for $k\ge 1$ and
\beq\label{EE} \bb(\vv):=\sup_{k\ge 1} \e^{-\vv \ll_k}q_k^2<\infty,\ \ \vv>0.\end{equation} Let $L_t$ be a L\'evy process on $\BB$ with L\'evy measure
$\nu$ satisfying (\ref{C}). Let $P_t$ be the Markov semigroup associated to the linear SDE
$$\d X_t= AX_t \d t + \d L_t.$$

\beg{prp} $(1)$ If $\nu_0(\BB)=\infty$, then $P_t$ is strong Feller for $t>0$.

$(2)$  If $\nu_0(\BB)<\infty$ and  there exist
$z_0\in\BB$ and $r_0>0$ such that $c_0:=\inf_{B(z_0,r_0)}\rr_0>0$,  then $(\ref{Coupling1})$ holds for
$$ \dd_2(\vv)= \ff 1 {c_0}\bigg[1+\exp\Big(\sup_{k\ge 1} q_k \e^{-2\vv \ll_k}\Big)\bigg]<\infty,\ \ \vv>0.$$ If, in particular, $q_k\approx k^{(1+\dd)}$ and $\ll_k\approx k^{2/d}$ for some constants $\dd, d>0$ and large $k$,
then there exists a constant $C>0$ such that
$$\|P_t(x,\cdot)-P_t(y,\cdot)\|_{var}\le \ff{C(1+\|x-y\|_\BB)}{t^{2/(4+d(1+\dd))}},\ \ \ t>0, x,y\in\BB.$$

$(3)$ Suppose $\ll:=\inf_{k\ge 1}\ll_k>0$. Then $(\ref{Z3})$ holds for any Lipschitz continuous $\rr_0$ with $\ll_0:=\nu_0(\BB)\in (0,\infty).$ \end{prp}

\beg{proof} Let $\H=\{h\in \BB: \sum_{k=1}^\infty h_k^2q_k^2<\infty\}.$ By (\ref{EE}) it is easy to check that {\bf (A)} holds for $\si=I$.
Moreover, by (\ref{EE0}), for any $h\in \H$ we have $\mu(\d z-h)=\varphi_h(z)\mu(\d z)$ for
\beq\label{EE1} \varphi_h(z)= \exp\bigg[ \sum_{k=1}^\infty \Big(q_k h_k z_k-\ff 1 2 q_k h_k^2\Big)\bigg],\ \ h_k=\<h,e_k\>, k\ge 1.\end{equation}
Then it is easy to see from (\ref{EE}) that there exists a constant $c_1>0$ such that
$$\sup_{\|y\|\le 1} \varphi_{T_s y}(z+T_s y)\le \exp\big[\|z\|_\BB^2 + c_1\bb(2s)\big]<\infty,\ \   z\in \BB.$$ Therefore, the first assertion follows from Theorem \ref{T1.1}.

Next, since $\si=I$, it follows from (\ref{EE1}) that
\beg{equation}\label{WWY2}\beg{split} \int_\BB \varphi_{\si^{-1}T_sy}(z)^2\mu(\d z) &= \prod_{k=1}^\infty \ff{\ss{q_k}}{\ss{2\pi}}\int_\R \exp\Big[q_k(T_sy)_k^2- \ff 12 q_k(z_k-2(T_sy)_k)^2\Big]\d z_k\\
&= \exp\bigg[\sum_{k=1}^\infty q_k (T_s y)_k^2\bigg]= \exp\bigg[\sum_{k=1}^\infty q_k \e^{-2\ll_k s} y_k^2\bigg]\\
&\le \exp\Big[\|y\|_\BB^2 \sup_{k\ge 1} q_k \e^{-2\ll_k s}\Big].\end{split}\end{equation} Thus, due to (\ref{EE}), Theorem \ref{T1.2} holds for the claimed
 $\dd_2(\vv)$.

Finally,  under \eqref{EE}, we have $\sup_{k\ge 1}q_k \e^{-s\ll_k}<\infty$ for $s>0$, which implies (\ref{Z2}). Moreover, replacing $T_s y$ by $h$ in (\ref{WWY2})   we obtain
$$\mu(\vp_h^2)-1 =\exp\Big[\sum_{k\ge 1}  q_k h_k^2\Big] -1\le \e^{c_2\|h\|_\H^2}-1\le \e^{c_2} \|h\|_\H^2,\ \ \|h\|_\H\le 1$$ for some constant $c_2>0.$ Then as in the proof of Proposition \ref{P4.1} we prove (\ref{Z1}).
\end{proof}

\paragraph{Acknowledgement.} The authors would like to thank the referee as well as an anonymous expert for very useful comments  and corrections.

\beg{thebibliography}{99}
\bibitem{App0} D. Applebaum, \emph{Martingale-valued measures, Ornstein-Uhlenbeck processes with jumps and operator self-decomposability in Hilbert space,} S\'{e}minaire de Probabilit\'{e}s 39, 171--196. Lect. Notes in Math. 1874, Springer, Berlin, 2006.

\bibitem{App1} D. Applebaum, \emph{L\'{e}vy processes and stochastic integrals in Banach spaces,} Probab. Math. Stat. 27(2007), 75--88.

\bibitem{BSW} B. B\"ottcher, R. L. Schilling, J. Wang, \emph{Constructions of coupling processes for L\'evy processes,}   Stoch. Proc. Appl.
121(2011), 1201--1216.

 \bibitem{CM} A. Chojnowska-Michalik, \emph{On processes of Ornstein-Uhlenbeck type in Hilbert space,} Stochastics 21(1987), 251--286.

 \bibitem{CG} M. Cranston, A. Greven, \emph{Coupling and harmonic functions in the case of continuous time Markov processes,} Stoch. Proc. Appl. 60(1995), 261--286.

\bibitem{DZ}  G. Da Prato, J. Zabczyk, \emph{Stochastic Equations in Infinite Dimensions,} Cambridge University Press,
Cambridge, 1992.

\bibitem{BN} B. Goldys, J. M. A. M. van Neerven, \emph{Transition semigroups of Banach space-valued Ornstein-Uhlenbeck processes,}
 Acta Appl. Math.  76(2003), 283--330.

 \bibitem{KS} V. Knopova, R.L. Schilling, \emph{A note on the existence of transition probability densities for L\'evy processes,}
  Forum Math. 25(2013), 125--149.

 \bibitem{KS2} V. Knopova, R.L. Schilling, \emph{Transition  density estimates for a class of L\'evy and   L\'evy-type processes,}
J. Theor. Probab. 25(2012), 144--170.

\bibitem{Lin} T. Lindvall, \emph{Lectures on the Coupling Methods,}  Wiley, New York, 1992.

\bibitem{Ma}  P. Malliavin, \emph{Stochastic Analysis,} Springer,
Berlin, 1997.

\bibitem{PZ1}  S. Peszat, J. Zabczyk, \emph{Stochastic Partial Equations with L\'{e}vy Noise,} Cambridge University Press,
Cambridge, 2007.

 \bibitem{X2} E. Priola, A. Shirikyan, L. Xu, J. Zabczyk, \emph{Exponential ergodicity and regularity for equations with L\'evy noise,} Stoch. Proc. Appl. 122(2012), 106--133.

\bibitem{PZ} E. Priola, J. Zabczyk, \emph{Densities for Ornstein-Uhlenbeck processes with jumps,} Bull. Lond. Math. Soc. 41(2009), 41--50.

\bibitem{X1} E. Priola, J. Zabczyk, \emph{Structural properties of semilinear SPDEs driven by cylindrical stable processes,} Probab. Theory Relat. Fields
149(2011), 97--137.

\bibitem{RW03} M. R\"ockner, F.-Y. Wang, \emph{Harnack and functional inequalities for generalized Mehler semigroups,} J. Funct. Anal. 203(2003), 237--261.

\bibitem{SW}  R. L. Schilling, J. Wang, \emph{On the coupling property of L\'evy processes,} Inst. Henri Poinc. Probab. Stat. 47(2011), 1147--1159.

\bibitem{SW2}  R. L. Schilling, J. Wang, \emph{On the coupling property and the Liouville theorem for Ornstein-Uhlenbeck processes,} J. Evol. Equat. 12(2012), 119--140.

\bibitem{SSW} R. L. Schilling,  P. Sztonyk, J. Wang, \emph{Coupling property and gradient estimates of L\'evy processes via the symbol,}   Bernoulli. 18(2012), 1128--1149.

\bibitem{T} A. Takeuchi, \emph{The Bismut-Elworthy-Li type formulae for stochastic differential equations with jumps, } J. Theor. Probab. 23(2010), 576--604.

\bibitem{W10b} F.-Y. Wang, \emph{Gradient estimate for Ornstein-Uhlenbeck jump processes,}  Stoch. Proc. Appl. 121(2011), 466--478.

\bibitem{W10a} F.-Y. Wang, \emph{Coupling for  Ornstein-Uhlenbeck   processes with jumps,}  Bernoulli 17(2011), 1136--1158.

\bibitem{W11} F.-Y. Wang, \emph{Derivative formula and Harnack inequality  for jump processes,} preprint,  arXiv: 1104.5531v4.

\bibitem{WangBook}
F.-Y. Wang, \emph{Functional Inequalities, Markov Processes and
Spectral Theory}, Science Press, Beijing 2005.

\bibitem{WJ} J. Wang, \emph{Linear evolution equations with cylindrical L\'{e}vy noise: gradient estimates and exponential ergodicity,} preprint.

\bibitem{Z} J. Zabczyk, \emph{Linear stochastic systems in Hilbert spaces; spectral properties and limit behaviour,}
Banach Center Pub. 41(1985), 591--609.

\bibitem{Zhang} X. Zhang, \emph{Derivative formula and gradient estimate for SDEs driven by $\aa$-stable processes,} preprint, arXiv:1204.2630.
  \end{thebibliography}
\end{document}